\documentclass[a4paper, 11pt]{article}
\usepackage{bbm, booktabs, threeparttable}
\usepackage[margin=1in]{geometry}
\usepackage[linesnumbered,ruled,vlined]{algorithm2e}
\usepackage{amsfonts, bm, enumerate}
\usepackage{amsmath,amsthm,amssymb}

\newtheorem{theorem}{Theorem}

\newtheorem{lemma}{Lemma}
\newtheorem{proposition}{Proposition}

\newtheorem{assumption}{Assumption}

{\theoremstyle{THkey}\newtheorem{restatetheorem}{Theorem}}

\usepackage{mathtools,xfrac}
\usepackage{color,xcolor}

\usepackage{authblk}

\usepackage{pdfsync,array}
\usepackage[normalem]{ulem}
\usepackage{float}
\usepackage{verbatim}
\usepackage{url}

\usepackage{multirow, multicol}
 
\usepackage{graphicx}
\usepackage{epsfig}
\usepackage{subfigure}
\usepackage{apacite}
\usepackage{natbib}
\usepackage[colorlinks=true,citecolor=blue]{hyperref}
\usepackage[capitalize]{cleveref}
\crefformat{equation}{(#2#1#3)}
\crefrangeformat{equation}{(#3#1#4)--(#5#2#6)}
\crefname{assumption}{Assumption}{Assumptions}
\Crefname{assumption}{Assumption}{Assumptions}

\def\minimize{\mathop{\rm \text{minimize}}}%
\def\argmax{\mathop{\rm \text{argmax}}}%
\def\argmin{\mathop{\rm \text{argmin}}}%

\def\D{\mathcal{D}}

\newcommand{\cS}{\mathcal{S}}

\def\X{\mathcal{X}}
\def\Z{\mathcal{Z}}
\def\B{\mathcal{B}}
\def\Q{\mathcal{Q}}

\def\ie{\emph{i.e.}, }

\newcommand{\BSP}{\textbf{BSP}}
\newcommand{\EBSP}{\textbf{EBSP}}
\newcommand{\RBSP}{\textbf{RBSP}}

\title{An Exact Algorithm for Mixed-Integer Bilevel Stochastic Problem}
\author[1]{Tom\'as Lagos\thanks{tomas.lagos@unsw.edu.au}}
\author[2]{Dmytro Matsypura\thanks{dmytro.matsypura@sydney.edu.au}}
\affil[1]{University of New South Wales}
\affil[2]{The University of Sydney}

\begin{document}
\maketitle

\begin{abstract}%
We study a class of mixed-integer bilevel stochastic programs in which the leader commits to a first-stage decision before uncertainty is realized, and the follower solves a mixed-integer optimization problem for each revealed scenario. Due to the hierarchical structure and discrete variables at both levels, these problems are inherently $\Sigma_2^p$-hard, rendering standard single-level reformulations computationally intractable. To address this challenge, we develop an exact algorithm that combines deterministic value-function reformulations with scenario-wise decomposition. Specifically, we propose an extended single-level reformulation and a corresponding relaxation that enable scenario decomposition. We then introduce a stochastic subgradient cutting-plane scheme that dynamically generates follower optimality cuts and updates the Lagrange multipliers. We prove that, under boundedness assumptions, our algorithm converges in finite time to a global optimum and provides valid upper and lower bounds throughout its execution.

\end{abstract}%


\section{Introduction}

Bilevel optimization has many applications in areas such as machine learning, defense, natural-resource management, logistics, energy systems, and telecommunications \citep{dempe2020bileveloptimizationadvances}. In most of these settings, an \emph{upper-level} decision must be made ex-ante, before some uncertainty is realized, while a \emph{lower-level} decision maker optimizes their response ex-post, after the uncertainty is revealed. This leads naturally to the framework of \emph{bilevel stochastic programming}. Although some progress has been made on bilevel models under uncertainty \citep[see][and the references therein]{BECK2023survey}, the case where the lower-level problem contains both integer and continuous variables remains largely unaddressed. In contrast, deterministic bilevel mixed-integer optimization has seen significant algorithmic advances over the past two decades \citep{lozano2017avaluefunction, Tahernejad2020branch-and-cut, Ralphs2025MibS}. To fill this gap in the bilevel stochastic optimization literature, we propose an exact algorithm tailored for mixed-integer recourse.

Let $[m]$ denote the set $\{1,2,\ldots,m\}$. The class of bilevel stochastic programs we consider has the following form:
\begin{subequations}
    \begin{align}
        \minimize_{x} \quad & \sum_{s \in \cS} p_s \phi_s^u (x,z_s) && &\tag{\BSP}\label{eq:bilevel_stochastic_program} \\
        \text{subject to} \quad & h_j^{x,u}(x) \ge b_j^u, && j \in [m^u], \label{eq:upper_level_constraints} \\
                & x \in \mathbb{R}^{n_1^u} \times \mathbb{Z}^{n_2^u}, && \\
                & z_s \in \D_s(x), && s \in \cS. \label{eq:bilevel_feasibility}
    \end{align}
\end{subequations}
where $\D_s(x)$ is the set of optimal solutions of the $x$-parameterized problem
\begin{equation}\label{eq:BSP_follower_problem}
\begin{aligned}
\minimize_{z_s} \quad & \phi_s^l (x, z_s) \\
\text{subject to} \quad & h_{j,s}^{x,l} (x) +  h_{j,s}^{z,l} (z_s) \ge b_{j,s}^l, \quad j \in [m^l],\\
& z_s \ge 0, \; z_s  \in \mathbb{R}^{n_1^l} \times \mathbb{Z}^{n_2^l}.
\end{aligned}
\end{equation}
Here, the leader selects a deterministic decision $x$ (subject to deterministic feasibility constraints~\cref{eq:upper_level_constraints}), and after each scenario is revealed, the follower optimizes $z_s$. Constraint~\cref{eq:bilevel_feasibility} enforces both feasibility and optimality of the follower's response. We assume that the values $p_s$ are rational for all $s \in \cS$, and restrict ourselves to a cooperative tie-breaking rule, that is, among all equivalent optimal responses, the follower selects the one that minimizes the leader's objective. 

Our goal is to develop an exact algorithm to solve \ref{eq:bilevel_stochastic_program} and then establish its convergence and correctness. We begin by defining the following sets. Let $\X^1$ denote the set of extreme points of the set $\{x^1 \in \mathbb{R}^{n_1^u} \colon x^2 \in \mathbb{Z}^{n_2^u},  h_j^{x,u}(x^1 , x^2) \ge b_j^u, \, j \in [m^u] \}$. In other words, $\X^1$ is the set of extreme points of the continuous portion of the upper-level feasible set. Next, let
\begin{equation*}
    \X = \{  x = (x^1 , x^2) \colon x^1 \in \X^1, \, x^2 \in \mathbb{Z}^{n_2^u}, \, h_j^{x,u}(x^1 , x^2) \ge b_j^u, \, j \in [m^u] \}.
\end{equation*} 
Note that both $\X$ and $\X^1$ are finite by construction. Also, for each $s\in \cS$, we define the sets
\begin{equation*}
\Omega_s =  \left\{ (x,z_s) \colon \quad \begin{aligned}
h_{j,s}^{x,l} (x) +  h_{j,s}^{z,l} (z_s) \ge b_{j,s}^l, \; j \in [m^l],  \\ 
x \in \X, \; z_s \ge 0, \; z_s  \in \mathbb{R}^{n_1^l} \times \mathbb{Z}^{n_2^l} ,
\end{aligned}\right\}
\end{equation*} and $\Z_s = \{z_s \colon (x,z_s) \in \Omega_s \}.$ 
The set $\Omega_s$ contains all leader–follower pairs $(x,z_s)$ that satisfy the lower- and the upper-level constraints of scenario $s$, whereas $\Z_s$ collects all follower decisions that are feasible for at least one $x\in \X$. 

The following assumptions will be useful in our analysis.

\begin{assumption}\label{assumption:Omega_bounded}
    The sets $\Omega_s$ for all $s \in \cS$ and $\X$ are bounded. 
\end{assumption}

\begin{assumption}\label{assumption:discrete_hxl}
    Functions $h_{j,s}^{x,l} (x)$ are integer valued for all $j \in [m^l]$, $s \in \cS$ and $x \in \X$.
\end{assumption}

\cref{assumption:Omega_bounded} is satisfied for almost all practical applications, since real systems are typically bounded. It will help us prove the finite convergence of the proposed algorithm. 

\Cref{assumption:discrete_hxl} simplifies the exposition by allowing a reformulation that converts the strict inequality of the lower-level infeasibility condition ($b_{j,s}^l -  h_{j,s}^{z,l} (z_s) > h_{j,s}^{x,l}(x)$) into a non-strict inequality ($\lceil b_{j,s}^l -  h_{j,s}^{z,l} (z_s) \rceil - 1 \ge  h_{j,s}^{x,l} (x)$). However, this assumption can be relaxed by choosing a small $\delta>0$ such that $b_{j,s}^l -  h_{j,s}^{z,l} (z_s) \ge h_{j,s}^{x,l}(x) + \delta$ guarantees that $z_s$ is an infeasible response. Furthermore, since we assume that all values in \ref{eq:bilevel_stochastic_program} are rational, we can relax \cref{assumption:discrete_hxl} by scaling all of the constraints $h_{j,s}^{x,l} (x) +  h_{j,s}^{z,l} (z_s) \ge b_{j,s}^l$ in \eqref{eq:BSP_follower_problem} to obtain integer values for $h_{j,s}^{x,l} (x)$. Because \ref{eq:bilevel_stochastic_program} is $\Sigma_2^p$-hard (as discussed in \cref{subsection:computational_complexity}), this scaling does not affect the asymptotic worst-case size of the single-level reformulation.

Taken together, \cref{assumption:Omega_bounded} and \cref{assumption:discrete_hxl} (or its relaxation) guarantee that the relevant follower responses and their feasibility conditions can be represented exactly using a closed feasible region with finitely many constraints. These properties are essential for deriving an exact mixed-integer solution approach.

\subsection{Literature Review}

Algorithmic advances in deterministic mixed-integer bilevel linear programming (MIBLP) rely primarily on value-function reformulations and branch-and-cut frameworks. \cite{lozano2017avaluefunction} developed an exact algorithm based on the optimal-value-function reformulation, generating dynamic constraints to handle optimistic and pessimistic follower behaviors. This foundational concept has been extended into comprehensive branch-and-cut methods utilizing intersection cuts and bilevel-specific valid inequalities to manage upper-level infeasibility \citep{Fischetti2017newgeneralpurpose, Tahernejad2020branch-and-cut}. Recently, \cite{lefebvre2025dantzig} proposed the first single-level reformulation leveraging Dantzig-Wolfe decomposition to convexify the follower's problem without explicit reliance on the value function.

Bilevel optimization under uncertainty extends these hierarchical structures to stochastic environments, though the mixed-integer case remains a significant algorithmic bottleneck \citep{BECK2023survey}. A critical development in this domain is the generalized value function, parameterized by both the objective and right-hand side, which facilitates the exact solution of multi-follower and two-stage stochastic mixed-integer programs \citep{Tavasliouglu2019SolvingStochastic}. For models with specific structures, such as stochastic right-hand sides and exclusive integer lower level, specialized cutting-plane methods can achieve finite convergence \citep{Zhang2021BilevelInteger}.

Crucially, however, there is currently no exact scenario decomposition algorithmic framework capable of solving general multi-scenario stochastic bilevel programs, where the follower's problem includes a mix of continuous and discrete decisions. Our work directly addresses this methodological~need. 

\subsection{Computational Complexity}\label{subsection:computational_complexity}

Because every bilevel optimization problem comprises two interdependent optimization problems, an upper-level and a lower-level, governed by non-cooperating decision makers, its computational complexity fundamentally exceeds that of single-level optimization. In particular, even simple bilevel optimization problems are known to reside in the complexity class $\Sigma_2^p$ due to the presence of two alternating quantifiers (see \cite{Woeginger2021thetrouble} and references therein).

Consequently, the mixed-integer bilevel stochastic program we study is $\Sigma_2^p$-hard in general. This intractability is notably robust and persists under stringent structural restrictions. For example, even in settings where the upper- and lower-level decision makers share the same objective function with opposite signs, the problem remains $\Sigma_2^p$-hard (see, e.g., \cite{furini2019maximum, LAGOS2025bilevel}).

This classification has immediate algorithmic implications. Since mixed-integer linear programs (MILPs) are $NP$-hard in general, the $\Sigma_2^p$-hardness of \ref{eq:bilevel_stochastic_program} implies that it cannot admit a polynomial-size reformulation as a single-level MILP unless $NP = \Sigma_2^p$. Such an equality would imply a collapse of the polynomial hierarchy, which is widely regarded as highly unlikely. Therefore, the existence of a compact polynomial-size single-level reformulation of \ref{eq:bilevel_stochastic_program} is unlikely. This conclusion motivates the development of specialized algorithmic frameworks, such as the enumeration-based approach proposed herein, that explicitly account for the problem’s hierarchical structure.

\subsection{Contributions and Outline}

The primary contribution of this manuscript is the development of the exact scenario decomposition algorithmic framework for solving mixed-integer bilevel stochastic programs (\ref{eq:bilevel_stochastic_program}). As discussed in \cref{subsection:computational_complexity}, the inherent $\Sigma_2^p$-hardness of these problems precludes compact, polynomial-size single-level reformulations. By successfully bridging the methodological gap between deterministic value-function reformulations and stochastic scenario-wise decomposition, our work provides a framework for solving a more general class of multi-scenario hierarchical problems previously not considered. 

In practice, this theoretical result allows decision-makers to move beyond restrictive heuristics or more restrictive frameworks when dealing with hierarchical optimization under uncertainty. Because real-world applications inherently involve multiple uncertain future states and discrete recourse decisions, the size of the mathematical formulation can grow exponentially. Therefore, a scenario-wise scalable approach is critical to decouple this complexity, enabling the exact solution of large-scale stochastic instances that would otherwise be computationally prohibitive.

Our main theoretical result establishes that, despite the non-convexities introduced by the discrete lower-level variables, our proposed decomposition algorithm is guaranteed to converge to the true global optimum in a finite number of steps. More formally, we prove the following:

\begin{theorem}\label{theorem:correctness}
\cref{alg:SSCP} terminates after a finite number of iterations and produces a solution $x^\star$ such that \[
\sum_{s \in \mathcal{S}} p_s \, \phi_s^u(x^\star,z_s)
\;\leq\;
f^\star + \epsilon,\] where $(x^\star,z)$ satisfies \eqref{eq:upper_level_constraints}-\eqref{eq:bilevel_feasibility}, and $f^\star$ denotes the optimal value of \ref{eq:bilevel_stochastic_program}.
\end{theorem}

In essence, Theorem \ref{theorem:correctness} guarantees that the absolute optimality gap of the returned solution $x^\star$ does not exceed the predefined tolerance $\epsilon$. In particular, when $\epsilon = 0$, the gap is zero, implying that $x^\star$ is a global optimum.

We establish this result through a sequence of structured steps. First, we derive an exact, single-level extended reformulation of \ref{eq:bilevel_stochastic_program} by explicitly enumerating the follower's best-response feasible space and applying non-anticipativity constraints to the leader's variables across all scenarios. Recognizing the ``L-shaped'' structure of this reformulation, we introduce a relaxation that dualizes the non-anticipativity constraints using Lagrangian multipliers, enabling scenario-wise decomposition. We then analytically establish that this Lagrangian dual provides a valid lower bound and enforces consensus over the convexified feasible space. Building on this, we design a stochastic subgradient cutting-plane method that iteratively generates follower optimality cuts and updates the multipliers. We prove that the sequence of lower bounds converges to the optimal solution of the convexified problem. Finally, because the sets of leader decisions and follower responses are finite, we demonstrate that augmenting the subgradient updates with a primal exhaustion mechanism (such as no-good cuts) guarantees finite termination at a true consensus solution that satisfies all lower-level optimality conditions.


Furthermore, we demonstrate the practical applicability and computational efficiency of our proposed framework through numerical experiments on two problem classes: the Uncertain Bilevel Knapsack problem \citep{Xue2017Uncertain}, detailed in \cref{sec:ubkp}, and the Bilevel Fuel Treatment Planning problem \citep{LAGOS2025bilevel}, summarized in \cref{subsec:biftp_summary} and reported in full in Section~\ref{sec:biftp} of the online companion.
For both problem classes, we also benchmark our framework against the generalized value-function reformulation of \citet{Tavasliouglu2019SolvingStochastic} under a common time and memory budget. These computational studies confirm that the framework returns valid bounds for every instance, closes the optimality gap on small and moderately sized instances, and retains tight bounds on the largest instances for which the benchmark reformulation returns no bounds.

The remainder of this paper is structured as follows. \Cref{sec:preliminaries} details the mathematical foundation of our framework, including the exact single-level reformulation, its relaxation, and the formal statement of the stochastic subgradient cutting-plane algorithm. \Cref{sec:main_results} presents the proofs of our main theoretical results. \Cref{sec:ubkp,sec:biftp} describe the numerical experiments. Finally, \Cref{sec:concluding_remarks} concludes the paper and discusses future research directions.

\section{Preliminaries}\label{sec:preliminaries}

Using the above setting, we now present a single-level reformulation of our bilevel stochastic problem \BSP, followed by its relaxation and the description of the proposed algorithm.

\subsection{Single-Level Reformulation} 
\label{subsec:relaxation_exact_alg}

We say that the leader solution $x$ \emph{blocks} the follower solution $z_s$ if and only if  $b_{j,s}^l -  h_{j,s}^{z,l} (z_s) > h_{j,s}^{x,l} (x)$ for some $j \in [m^l]$. Let $\gamma_{j,z_s}^s = \lceil b_{j,s}^l -  h_{j,s}^{z,l} (z_s) \rceil - 1$. Following a similar line of reasoning as in \cite{lozano2017avaluefunction}, we note that $\gamma_{j,z_s}^s \geq h_{j,s}^{x,l} (x)$ if and only if the follower solution $z_s \in \Z_s$ for scenario $s \in \cS$ is an infeasible response to the leader solution $x$, \ie the leader solution $x$ blocks the follower solution $z_s$. To see that this result holds, observe that due to \cref{assumption:discrete_hxl}: \begin{align*}
    & \gamma_{j,z_s}^s \geq  h_{j,s}^{x,l} (x) & \Rightarrow & \quad\lceil b_{j,s}^l -  h_{j,s}^{z,l} (z_s) \rceil - 1 \geq  h_{j,s}^{x,l} (x)  & \Rightarrow & \quad b_{j,s}^l -  h_{j,s}^{z,l} (z_s) > h_{j,s}^{x,l} , \\
    & \gamma_{j,z_s}^s <  h_{j,s}^{x,l} (x) & \Rightarrow & \quad\lceil b_{j,s}^l -  h_{j,s}^{z,l} (z_s) \rceil - 1 <  h_{j,s}^{x,l} (x) & \Rightarrow & \quad b_{j,s}^l -  h_{j,s}^{z,l} (z_s) \leq h_{j,s}^{x,l} .
\end{align*}

Next, we note that for any pair of follower solutions $z_s,\hat{z}_s \in \Z_s$, if $\gamma_{j,z_s}^s \geq \gamma_{j,\hat{z}_s}^s$ and $x$ blocks $\hat{z}_s$, then $x$ must also block $z_s$. We leverage this property to derive an exact reformulation based on the exhaustive enumeration of the follower's feasible space $\Z_s$. Letting $\widehat{\Z}_{s}$ denote a subset of follower feasible solutions, for any $\hat{z}_s \in \Z_s$, we define:
\begin{equation*}
    \B_s(z_s,\widehat{\Z}_{s}) = \{ (\hat{z}_s,j) \colon \gamma_{j,z_s}^s \ge \gamma_{j,\hat{z}_s}^s, \, \hat{z}_s \in \widehat{\Z}_{s}, \, j \in [m^l] \}.
\end{equation*} The sets $\B_s(z_s,\widehat{\Z}_{s})$ contain the pairs of solutions and constraint indices $(\hat{z}_s,j)$ such that if $\hat{z}_s$ is blocked, then $z_s$ is also blocked by means of the $j$th lower-level constraint.

Using the notation and assumptions introduced above, we construct the following \emph{extended} reformulation of \ref{eq:bilevel_stochastic_program}:
\begin{subequations}\label{subeq:exponential_reformulation_BSP}
    \begin{align}
        \minimize_{\tilde{x},z,w} \quad &   \sum_{s \in \cS} p_s  \phi_s^u (\tilde{x}_s,z_s) && \tag{\EBSP} \label{eq:exponential_reformulation_BSP} \\
        \text{subject to}\quad & (\tilde{x}_s,z_s) \in \Omega_s, && s \in \cS, \label{eq:upper_lower_level_feasibility} \\
        & \tilde{x}_s = \sum_{\hat{s} \in \cS} p_{\hat{s}} \tilde{x}_{\hat{s}}, && s \in \cS, \label{eq:nonanticipativity_constraint_EBSP} \\
        &  h_{j,s}^{x,l} (\tilde{x}_s) \le M_{j,s} + \sum_{\hat{z}_s \in \Z_s} (\gamma_{j,\hat{z}_s}^s-M_{j,s}) w_{\hat{z}_s,j}^s, && \begin{array}{l} j \in [m^l], 
        s \in \cS, \end{array} \label{eq:follower_blocked_solutions} \\
        &                   \phi_s^l (\tilde{x}_s,z_s) \le \phi_s^l (\tilde{x}_s,\hat{z}_s) + \tilde{M}_{\hat{z}_s,s} \sum_{(\tilde{z}_s,j) \in \B_s(\hat{z}_s,\Z_s)} w_{\tilde{z}_s,j}^s, && \begin{array}{l} \hat{z}_s \in \Z_s , 
        \, s \in \cS, \end{array} \label{eq:bilevel_feasibility_cuts_BSP} \\
        &                    w_{\hat{z}_s,j}^s \in \{0,1\}, && \begin{array}{l} j\in[m^l], \\ \hat{z}_s \in \Z_s ,
        s \in \cS. \end{array}  \label{eq:binary_nature_w}
    \end{align}
\end{subequations}
In \ref{eq:exponential_reformulation_BSP}, the variable $x$ is augmented with a scenario index $s \in \cS$, effectively expanding the decision vector to account for each realization of uncertainty independently. We further require that both the leader and the follower decision variables be feasible in \cref{eq:upper_lower_level_feasibility} for each scenario. 
To ensure this reformulation is equivalent to \ref{eq:bilevel_stochastic_program}, we impose \cref{eq:nonanticipativity_constraint_EBSP}, known as the \emph{non-anticipativity constraint} in stochastic programming. It allows us to rewrite the first-stage variables as second-stage variables, \ie, as decisions made when all uncertain information is revealed to the leader. It is easy to check that the non-anticipativity constraint \cref{eq:nonanticipativity_constraint_EBSP} restricts the leader's solution to have the same value across all scenarios by checking
\begin{equation*}
\tilde{x}_s = \sum_{\hat{s} \in \cS} p_{\hat{s}} \tilde{x}_{\hat{s}} = \tilde{x}_{s'}, \qquad \forall s,s' \in \cS. 
\end{equation*}
In \cref{eq:follower_blocked_solutions} we impose that any solution blocked must be an infeasible response to the selected leader solution $\tilde{x}_s$. We model whether solution $\hat{z}_s$ is blocked using the binary variable $w$, that is $w_{\hat{z}_s,j}^s$ is equal to 1 if and only if $\hat{z}_s$ is blocked by solution $x$ by means of the $j$th constraint. 
Consequently, if no constraint blocks solution $\hat{z}_s$, then this solution must be considered in \cref{eq:bilevel_feasibility_cuts_BSP}, thereby imposing that $z_s$ is at least as attractive to the follower as $\hat{z}_s$.

The following result follows directly from \citet[Proposition 3]{lozano2017avaluefunction}:
\begin{proposition}\label{proposition:EBSP_equivalence_BSP}
    \ref{eq:exponential_reformulation_BSP} is equivalent to \ref{eq:bilevel_stochastic_program}.
\end{proposition}

In contrast to \citet{lozano2017avaluefunction}, we consider multiple scenarios at the lower level, each with its own independent follower. To achieve scalability across these scenarios, our method draws on the ideas from \citet{Lagos2025SCA}. 

\subsection{Relaxation}

Although \ref{eq:exponential_reformulation_BSP} is exact, the exponential number of \cref{eq:follower_blocked_solutions}-\cref{eq:binary_nature_w} constraints prevents direct implementation in most cases. However, only a small fraction of these constraints are typically active, which motivates an iterative approach that adds blocking and dominance conditions only as needed. Furthermore, relaxing \cref{eq:nonanticipativity_constraint_EBSP} allows us to decompose the problem by scenario, thereby requiring significantly fewer computational resources. This leads to the cutting-plane algorithm we develop next. We begin with the following \emph{relaxation} of \ref{eq:exponential_reformulation_BSP}:
\begin{subequations}\label{subeq:relaxed_bilevel_stochastic_program}
\begin{align}
    \minimize_{\tilde{x},z,w} \quad & \sum_{s \in \cS} p_s \left(\phi_s^u (\tilde{x}_s,z_s) + (\lambda_s)^\top \left(\tilde{x}_s - \sum_{\hat{s} \in \cS} p_{\hat{s}} \tilde{x}_{\hat{s}}  \right) \right)  && \tag{\RBSP} \label{eq:relaxed_bilevel_stochastic_program}\\
    \text{subject to} \quad & (\tilde{x}_s,z_s) \in \Omega_s,  && s \in \cS, \label{eq:upper_lower_level_feasibility_RBSP} \\
    & h_{j,s}^{x,l} (\tilde{x}_s) \le M_{j,s} + \sum_{\mathbf{\hat{z}} \in \widehat{\Z}_{s}} (\gamma_{j,\hat{z}_s}^s-M_{j,s}) w_{\hat{z}_s,j}^s, && j \in[m^l], \label{eq:follower_blocked_solutions_EBSP} \\
    & \phi_s^l (\tilde{x}_s,z_s) \le \phi_s^l (\tilde{x}_s,\hat{z}_s) + \tilde{M}_{\hat{z}_s,s} \sum_{(\tilde{z},j) \in \B_s(\hat{z}_s,\widehat{\Z}_{s})} w_{\tilde{z}_s,j}^s, && \begin{array}{l} \hat{z}_s \in \widehat{\Z}_{s} , 
    s \in \cS, \end{array} \label{eq:bilevel_feasibility_cuts_EBSP} \\
    & w_{\hat{z}_s,j}^s \in \{0,1\}, && \begin{array}{l} j\in [m^l], 
    \hat{z}_s \in \widehat{\Z}_{s} , 
    s \in \cS.\end{array} \label{eq:binary_nature_w_EBSP}
\end{align}
\end{subequations}
\ref{eq:relaxed_bilevel_stochastic_program} is a subsystem of \ref{eq:exponential_reformulation_BSP} whose objective penalizes violations of the non-anticipativity constraint \cref{eq:nonanticipativity_constraint_EBSP}, with $\lambda_s \in \mathbb{R}^{n_1^u + n_2^u}$ denoting the corresponding Lagrange multipliers. Constraint \cref{eq:upper_lower_level_feasibility_RBSP} coincides with \cref{eq:upper_lower_level_feasibility}, while \eqref{eq:follower_blocked_solutions_EBSP}-\eqref{eq:bilevel_feasibility_cuts_EBSP} are relaxations of \eqref{eq:follower_blocked_solutions}-\eqref{eq:bilevel_feasibility_cuts_BSP} obtained by retaining only a subset $\widehat{\Z}_{s}$ of follower solutions, i.e., $\widehat{\Z}_{s} \subseteq \Z_s$. 

We use $\Z$ to denote the Cartesian product of all sets $\Z_s$, i.e., $\Z = \prod_{s\in \cS} \Z_s$, and analogously define $\widehat{\Z}=\prod_{s\in \cS}\widehat{\Z}_{s}$. Let $f^\star$, $f_e^\star$, and $f_r^\star(\widehat{\Z},\cS,\lambda)$ denote the optimal values of \ref{eq:bilevel_stochastic_program}, \ref{eq:exponential_reformulation_BSP}, and \ref{eq:relaxed_bilevel_stochastic_program}, respectively. Note that the optimal value of \ref{eq:relaxed_bilevel_stochastic_program} is parametrized by the follower feasible set $\widehat{\Z}$, the set of scenarios $\cS$, and the vector of Lagrange multipliers $\lambda = (\lambda_1,\ldots,\lambda_{\vert \cS \vert})$. The following result shows that \ref{eq:relaxed_bilevel_stochastic_program} is indeed a relaxation of  \ref{eq:bilevel_stochastic_program}, and therefore its objective value is a valid lower bound for $f^\star$.

\begin{proposition}\label{proposition:lower_bound}
    Let $\lambda \in \mathbb{R}^{\vert\cS\vert(n_1+n_2)}$ and $\widehat{\Z} \subseteq \Z$. Then $f_r^\star(\widehat{\Z},\cS,\lambda) \leq f^\star$.
\end{proposition}
\begin{proof}
    Suppose $(\tilde{x}^\star,z^\star)$ is an optimal solution to \ref{eq:exponential_reformulation_BSP}. Then,
    \begin{align*}
        f_r^\star(\widehat{\Z},\cS,\lambda)  & =  \min_{\tilde{x},z \colon \eqref{subeq:relaxed_bilevel_stochastic_program}} \sum_{s \in \cS} p_s \left(\phi_s^u (\tilde{x}_s,z_s) + (\lambda_s)^\top \left(\tilde{x}_s - \sum_{\hat{s} \in \cS} p_{\hat{s}} \tilde{x}_{\hat{s}}  \right) \right) \\
        & \le \sum_{s \in \cS} p_s \left(\phi_s^u (\tilde{x}_s^{\star}, z_s^{\star}) + (\lambda_s)^\top \left(\tilde{x}_s^{\star} - \sum_{\hat{s} \in \cS} p_{\hat{s}} \tilde{x}_{\hat{s}}^{\star}  \right) \right) \\
        & = \sum_{s \in \cS} p_s \left(\phi_s^u (\tilde{x}_s^{\star}, z_s^{\star}) \right) = f_e^\star.
    \end{align*}
    By \cref{proposition:EBSP_equivalence_BSP}, the result $f_r^\star(\widehat{\Z},\cS,\lambda) \le f^\star$ follows. 
\end{proof}

\subsection{Stochastic Subgradient Cutting Plane Algorithm}\label{subsec:CPM_exact_alg}

Our Stochastic Subgradient Cutting Plane algorithm, which we refer to as \cref{alg:SSCP}, combines stochastic subgradient updates with cut generation. At the beginning of each iteration, given the current multipliers $\lambda_s$ and the sets of follower responses $\widehat{\Z}_{s}$, we first solve \ref{eq:relaxed_bilevel_stochastic_program} to obtain scenario-wise solutions $(\hat{x}_{s}, \hat{z}_{s})$ and update the lower bound $LB$. We then compute the subgradient $\hat{\chi}$ from \cref{lemma:subgradient_REBSP} and add it to a local subgradient bundle $\mathcal{B}$. A consensus rule (\cref{algstep:consensus}) selects a representative solution $x$ from these scenario-wise solutions. If $x$ is new ($x \notin \widehat{\X}$), it is added to $\widehat{\X}$ and we compute its follower best responses $\tilde{z}_{s}$ by solving the lower-level problem~\cref{eq:BSP_follower_problem} for each scenario. We then update the sets $\widehat{\Z}_{s}$ for the next iteration, tighten the upper bound $UB$, reset the subgradient step size $\eta$, and clear the bundle $\mathcal{B}$.

Otherwise, if $x$ has already been evaluated, we check for convergence. If $\bar{x} = x$, discrete consensus is achieved, $UB$ meets $LB$, and the optimal solution is found. If not, as is standard in Lagrangian relaxation for mixed-integer programs, the subgradient updates will eventually exhibit a \emph{tailing-off} effect, where the lower bound improvements stall while the discrete scenario solutions oscillate due to the nonconvex duality gap \citep{CAROE1999Dual,Guignard2003lagrangian}. Rather than relying on heuristic tolerance triggers, \cref{alg:SSCP} detects this stalling by evaluating a theoretical optimality certificate via a bundle method framework. It solves a restricted quadratic program to find the minimum-norm element in the convex hull of the historical subgradients tracked in $\mathcal{B}$. If this norm evaluates to zero ($0 \in \partial_\lambda f_r^\star(\widehat{\Z}, \cS, \lambda)$), it formally certifies that the dual updates have stalled at the optimum \citep{Lemarechal2001Lagrangian}. At this point, a primal exhaustion mechanism is invoked: the algorithm generates a no-good cut to explicitly exclude the current stalled solution from the leader's feasible set ($\X \gets \X \setminus \{x\}$), resets the step size, and clears the bundle $\mathcal{B}$, thereby resolving the remaining duality gap and forcing exact convergence. If the dual has not yet stalled, only the step size $\eta$ is updated. In all cases, the multipliers $\lambda_s$ are updated, and the process is repeated until the absolute optimality gap $(UB-LB)$ is below $\epsilon$ or consensus has been reached.

\begin{algorithm}[ht]\small
\caption{Stochastic Subgradient Cutting Plane Algorithm}\label{alg:SSCP}
\SetAlgoLined
    \KwResult{$\epsilon$-optimal solution for \ref{eq:bilevel_stochastic_program}}
    \KwIn{Feasible $x^0$, set of scenarios $\cS$, tolerance $\epsilon > 0$}
    \textbf{Initialize:} For each $s\in \cS$, solve the follower problem~\eqref{eq:BSP_follower_problem} parametrized by $x^0$ and use solution $z_s$ to construct feasible leader-follower pairs $(x^0, z_s)\in\Omega_s$\;
    Set $\lambda_s \gets 0$, $\widehat{\Z}_{s} \gets \{ z_s \}$, $x^\star \gets x^0$ and $\widehat{\X} \gets \{x^\star \}$\;
    Set $\eta \gets 1$, $\B \gets \emptyset$, $UB \gets \sum_{s \in \cS} p_s  \phi_s^u (x,z_s)$ and $LB \gets -\infty$\;
    \While{$UB-LB > \epsilon$\label{alg_step:while_condition}}{    
    Obtain $(\hat{x}_s,\hat{z}_s)$ by computing $f_r^\star(\widehat{\Z}_{s},\{s\},\lambda_s)$ for each $s \in \cS$\label{alg_step:solve_relaxation}\; 
    \If{$ LB < \sum_{s \in \cS} f_r^\star(\widehat{\Z}_{s},\{s\},\lambda_s) $}{
	$LB\gets \sum_{s \in \cS} f_r^\star(\widehat{\Z}_{s},\{s\},\lambda_s)$\;
    }
    Compute $\bar{x} = \sum_{s \in \cS} p_s \hat{x}_s$ \;  
    Compute subgradient $\hat{\chi}$ using \cref{lemma:subgradient_REBSP}, and update bundle $\mathcal{B} \gets \mathcal{B} \cup \{\hat{\chi}\}$\;
    Obtain $s^\star = \displaystyle \arg\min_{s \in \cS} \sum_{s' \in \cS \setminus \{s\}} \| \hat{x}_s - \hat{x}_{s'} \|_1 $, and set $x \gets \hat{x}_{s^{\star}}$\label{algstep:consensus}\;
        \uIf{$x \notin \widehat{\X}$ \label{algstep:CPM_non_recorded_leader_sol} }{
            $\widehat{\X} \gets \widehat{\X} \cup \{x\}$\; 
            For each $s \in \cS$, compute $\tilde{z}_s$ by solving the follower problem~\eqref{eq:BSP_follower_problem} parametrized by $x$\label{algstep:optimal_follower_response}\;
            \If{$\phi_s^l (x,\tilde{z}_s)<\phi_s^l (x,\hat{z}_s)$ for any $s \in \cS$}{ 
                $\pi^\star \gets 0$\;
                \For{$s \in \cS$}{
                    \uIf{$\phi_s^l (x,\tilde{z}_s)<\phi_s^l (x,\hat{z}_s)$}{
                        $\widehat{\Z}_{s} \gets \widehat{\Z}_{s} \cup \{ \tilde{z}_s \}$\;
                        $\pi^\star \gets \pi^\star + p_s \phi_s^u (x,\tilde{z}_s)$\;
                    }
                    \Else{
                        $\pi^\star \gets \pi^\star +p_s \phi_s^u (x,\hat{z}_s)$\;
                    }
                }
                \If{$UB > \pi^\star$}{
                    $UB \gets \pi^\star$ and $x^\star \gets x$\;
                }
            }
            Reset step size $\eta \gets 1$ and clear bundle $\mathcal{B} \gets \emptyset$ \label{algstep:reset_eta}\;
        }\uElseIf{$\bar{x}=x$}{
                $x^\star \gets x$ and break\label{algstep:set_bilevel_optimal_sol_2}\;
            }
            \uElseIf{$\displaystyle \min_{\alpha \ge 0, \sum \alpha_{k=1}^{\vert \B \vert} = 1} \left\| \sum_{\hat{\chi}^k \in \mathcal{B}} \alpha_k \hat{\chi}^k \right\|_2^2 = 0$}{
                Add a no-good cut to exclude the current solution $x$: $\X \gets \X \setminus \{x\}$\label{algstep:add_no_good_cut}\;
                Reset step size $\eta \gets 1$ and clear bundle $\mathcal{B} \gets \emptyset$\;
            }
            \Else{
                Update step size $\eta \gets \eta / (1 + \eta)$ \label{algstep:update_eta}\;
            }
        Update $\lambda_s \gets \lambda_s + \eta (x - \bar{x})$ \label{algstep:update_lambda}\;
    }
    \textbf{Return:} ${x}^\star$\;
\end{algorithm}

\section{Proof of the Main Result}\label{sec:main_results}

To prove \cref{theorem:correctness}, we first show that there exists a vector of Lagrange multipliers $\lambda$ such that the corresponding relaxation satisfies the non-anticipativity constraint and attains the optimal value of~\ref{eq:bilevel_stochastic_program}.

\begin{proposition}\label{proposition:non_anticipativity_convergence}
    
Let $LB = \max_{\lambda} f_r^\star(\widehat{\Z},\cS,\lambda)$ denote the optimal value of the Lagrangian dual of \ref{eq:relaxed_bilevel_stochastic_program}. Then $LB$ is equal to the optimal value of the convex relaxation of \ref{eq:exponential_reformulation_BSP} with the non-anticipativity constraint explicitly enforced, and hence $LB \leq  f_e^\star$.
\end{proposition}
\begin{proof}
    Define the discrete feasible set of the relaxed reformulation as $\overline{\Omega} = \{ (\tilde{x},z) \colon \eqref{subeq:relaxed_bilevel_stochastic_program} \}.$ 
    By weak duality, the maximum of the Lagrangian relaxation provides a valid lower bound for the primal problem. 
    Furthermore, by \cref{assumption:Omega_bounded}, $\overline{\Omega}$ is bounded. 
    Let $\text{conv}(\overline{\Omega})$ denote the convex hull of $\overline{\Omega}$. 
    Because the objective function is linear in both the decision variables and the multipliers, we can apply Sion's minimax theorem \citep{sion1958general} to the convexified space to interchange the operators:
    \begin{align}\label{eq:interchanged_max_min_lagrangian_problem}
        \max_{\lambda}{ f_r^\star(\widehat{\Z},\cS,\lambda) } & = \max_{\lambda} \left\{ \min_{(\tilde{x},z) \in \overline{\Omega}} \sum_{s \in \cS} p_s \left(\phi_s^u (\tilde{x}_s,z_s) + (\lambda_s)^\top \left(\tilde{x}_s - \sum_{\hat{s} \in \cS} p_{\hat{s}} \tilde{x}_{\hat{s}}  \right) \right) \right\} \nonumber \\
        & = \max_{\lambda} \left\{ \min_{(\tilde{x},z) \in \text{conv}(\overline{\Omega})} \sum_{s \in \cS} p_s \left(\phi_s^u (\tilde{x}_s,z_s) + (\lambda_s)^\top \left(\tilde{x}_s - \sum_{\hat{s} \in \cS} p_{\hat{s}} \tilde{x}_{\hat{s}}  \right) \right) \right\} \nonumber \\
        & = \min_{(\tilde{x},z) \in \text{conv}(\overline{\Omega})} \left\{ \max_{\lambda} \sum_{s \in \cS} p_s \left(\phi_s^u (\tilde{x}_s,z_s) + (\lambda_s)^\top \left(\tilde{x}_s - \sum_{\hat{s} \in \cS} p_{\hat{s}} \tilde{x}_{\hat{s}}  \right) \right) \right\}.
    \end{align}
If the non-anticipativity constraint $\tilde{x}_s - \sum_{\hat{s} \in \cS} p_{\hat{s}} \tilde{x}_{\hat{s}} = 0$ is violated for any $s \in \cS$ in the convexified space, the inner maximization over $\lambda$ diverges to $+\infty$. 
Consequently, any finite optimal solution to the right side of \cref{eq:interchanged_max_min_lagrangian_problem} must satisfy non-anticipativity in $\text{conv}(\overline{\Omega})$. Thus, solving the Lagrangian dual yields a lower bound on the convexified primal problem over $\text{conv}(\overline{\Omega})$ subject to non-anticipativity. 
\end{proof}

\cref{proposition:non_anticipativity_convergence} implies that, for an appropriate choice of optimal multipliers $\lambda^\star$ and a sufficient number of follower responses in \cref{eq:bilevel_feasibility_cuts_BSP}, solving \ref{eq:relaxed_bilevel_stochastic_program} solves the convex relaxation of \ref{eq:exponential_reformulation_BSP}. Specifically, due to the presence of discrete decision variables, evaluating the Lagrangian dual at $\lambda^\star$ yields the optimal value of the problem over the convex hull of the scenario-wise feasible sets. In other words, the non-anticipativity constraint is enforced over a convex combination of optimal scenario-wise discrete solutions, rather than on a single discrete consensus point.

A vector $u \in\mathbb{R}^{n}$ is a \emph{subgradient} of a concave function $g(x):\mathbb{R}^{n} \to \mathbb{R}$ at $x_0$, if for all $x \in \mathbb{R}^{n}$ 
\begin{equation*}
    g(x) \leq g(x_0) + u^\top (x - x_0).
\end{equation*} 

Let $\bar{x} = \sum_{s \in \cS} p_s \hat{x}_s$. Then, we have the following result:
\begin{lemma}\label{lemma:subgradient_REBSP}
Let $(\hat{x},\hat{z})$ be an optimal solution to \ref{eq:relaxed_bilevel_stochastic_program} with parameters $(\widehat{\Z},\cS,\hat{\lambda})$. Then, a subgradient of $f_r^\star(\widehat{\Z},\cS,\lambda)$ at $\hat{\lambda}$ is given by     
    \begin{align}
        \hat{\chi} &= \left(p_1 \hat{x}_1 - p_1\sum_{\hat{s} \in \cS} p_{\hat{s}} \hat{x}_{\hat{s}}, \ldots, p_{|\cS|} \hat{x}_{|\cS|} - p_{|\cS|}\sum_{\hat{s} \in \cS} p_{\hat{s}} \hat{x}_{\hat{s}} \right) \nonumber\\
        &= (p_1(\hat{x}_1 - \bar{x}),\ldots, p_{|\cS|}(\hat{x}_{|\cS|} - \bar{x})).\nonumber
    \end{align}
\hfill\end{lemma}
\begin{proof}
    The statement follows from the fact that the function
    \begin{align*}
        f_r^\star(\widehat{\Z},\cS,\lambda)  & =  \min_{\tilde{x},z \colon \eqref{subeq:relaxed_bilevel_stochastic_program}} \left\{ \sum_{s \in \cS} p_s \left(\phi_s^u (\tilde{x}_s,z_s) + (\lambda_s)^\top \left(\tilde{x}_s - \sum_{\hat{s} \in \cS} p_{\hat{s}} \tilde{x}_{\hat{s}}  \right) \right) \right\} \\
        & = \sum_{s \in \cS} p_s \left(\phi_s^u (\hat{x}_s,\hat{z}_s) + (\lambda_s)^\top \left(\hat{x}_s - \sum_{\hat{s} \in \cS} p_{\hat{s}} \hat{x}_{\hat{s}}  \right) \right)
    \end{align*}
    is concave in $\lambda$; see \cite{boyd_vandenberghe2004convex}. We also note that $\partial f_r^\star(\widehat{\Z},\cS,\lambda) / \partial\lambda =  \hat{\chi}$. By concavity of $f_r^\star$, it follows that at $\hat{\lambda}$ we have:
    \begin{equation*}
        f_r^\star(\widehat{\Z},\cS,\lambda) \leq f_r^\star(\widehat{\Z},\cS,\hat{\lambda}) + \hat{\chi}^\top (\lambda - \hat{\lambda}).
    \end{equation*}
\end{proof}

Because the Lagrangian dual function is non-smooth, optimal multipliers are theoretically identified when the zero vector belongs to the subdifferential, i.e., $0 \in \partial_\lambda f_r^\star(\widehat{\Z},\cS,\lambda^\star)$ \citep{Rockafellar1970ConvexAnalysis}. As detailed in \cref{subsec:CPM_exact_alg}, \cref{alg:SSCP} actively monitors this exact mathematical certificate to detect when the multiplier updates have stalled at the dual optimum \citep{Lemarechal2001Lagrangian}. Recognizing this stall is precisely what triggers the primal exhaustion mechanism, enabling the algorithm to overcome the nonconvex duality gap and proceed toward an exact consensus.

\begin{theorem}\label{theorem:convergence}
Let $LB^\nu$ denote the lower bound obtained by \cref{alg:SSCP} at iteration $\nu$, i.e., 
\[
LB^\nu = \max_{k\le \nu} \sum_{s \in \cS} f_r^\star(\widehat{\Z}_s^{k},\{s\},\lambda_s^{k}).
\] 
Then 
\[
LB^\nu \rightarrow f_r^\star(\widehat{\Z}^{\nu_0},\cS,\lambda^{\star}) \quad \text{as} \quad \nu\rightarrow\infty,
\] 
where $\lambda^{\star} \in \argmax_{\lambda} f_r^\star(\widehat{\Z}^{\nu_0},\cS,\lambda)$ and $\nu_0$ is the iteration number, at which set $\widehat{\X}$ in \cref{alg:SSCP} stops changing.
\end{theorem}
\begin{proof}
Since $\widehat{\X} \subseteq \X$ and set $\X$ is finite, $\widehat{\X}$ must also be finite. Hence, after a finite number of iterations (denoted by $\nu_0$) in \cref{alg:SSCP}, set $\widehat{\X}$ stops changing. 
    
    Let $\lambda^{\star} \in \argmax_{\lambda} f_r^\star(\widehat{\Z}^{\nu_0},\cS,\lambda)$. We can establish the following relationships:
    \begin{align}
        \sum_{s \in \cS}  p_s \| \lambda_s^{\nu} - \lambda_s^{\star} \|_2^2 &= \sum_{s \in \cS} p_s \| \lambda_s^{\nu-1} + \eta^{\nu} (x_s^{\nu} - \bar{x}^\nu) - \lambda_s^{\star} \|_2^2 \nonumber \\
        & \begin{aligned} \; \le & \sum_{s \in \cS} p_s \left( \| \lambda_s^{\nu-1} - \lambda_s^{\star} \|_2^2 +  (\eta^{\nu})^2 \|  x_s^{\nu} - \bar{x}^\nu  \|_2^2 \right) \\ 
        & \qquad\qquad  +  2 \eta^{\nu} (f_r^\star(\widehat{\Z}^{\nu-1},\cS,\lambda^{\nu-1})  - f_r^\star(\widehat{\Z}^{\nu-1},\cS,\lambda^{\star}) ) \end{aligned} \nonumber\\
        & \begin{aligned} \; \le & \sum_{s \in \cS} p_s \left( \|\lambda_s^{\nu_0} - \lambda_s^{\star} \|_2^2 + \sum_{k=\nu_0}^\nu (\eta^{k})^2 \|  x_s^{k} - \bar{x}^k  \|_2^2 \right)  \\ 
        & \qquad\qquad +  2 \sum_{k=\nu_0}^{\nu-1}  \eta^{k} (f_r^\star(\widehat{\Z}^{k},\cS,\lambda^{k})  - f_r^\star(\widehat{\Z}^{k},\cS,\lambda^{\star}) )
                            \end{aligned} \nonumber \\
        & \begin{aligned} \;= & \sum_{s \in \cS} p_s \left( \|\lambda_s^{\nu_0} - \lambda_s^{\star} \|_2^2 + \sum_{k=\nu_0}^\nu (\eta^{k})^2 \|  x_s^{k} - \bar{x}^k  \|_2^2 \right)  \\ 
        & \qquad\qquad +  2 \sum_{k=\nu_0}^{\nu-1}  \eta^{k} (f_r^\star(\widehat{\Z}^{\nu_0},\cS,\lambda^k)  - f_r^\star(\widehat{\Z}^{\nu_0},\cS,\lambda^{\star}) )
                            \end{aligned} \nonumber \\
        & \begin{aligned} \;\le & \sum_{s \in \cS} p_s \left( \|\lambda_s^{\nu_0} - \lambda_s^{\star} \|_2^2 + \sum_{k=\nu_0}^\nu (\eta^{k})^2 \|  x_s^{k} - \bar{x}^k  \|_2^2 \right) \\ 
        & \qquad\qquad  +  2 (LB^\nu
        - f_r^\star(\widehat{\Z}^{\nu_0},\cS,\lambda^{\star}) ) \sum_{k=\nu_0}^{\nu-1} \eta^{k}. \end{aligned} \label{eq:inequality_distance_lambdanu_lambdastar_2}
    \end{align} 
    The first equality follows from substituting the expression for $\lambda_s^{\nu}$ (line \ref{algstep:update_lambda}, \cref{alg:SSCP}). The first and second inequalities follow directly from~\cref{lemma:subgradient_REBSP}. The last equality is due to $\widehat{\Z}^{k}=\widehat{\Z}^{\nu_0}$ for all $k \geq \nu_0$. To establish the last inequality, we use the fact that for any fixed $k$, 
    \begin{equation*}
    f_r^\star (\widehat{\Z}^{\nu_0}, \cS, \lambda^k) \leq \max_{k\le \nu} f_r^\star(\widehat{\Z}^{k},\cS,\lambda^{k}) = LB^\nu.
    \end{equation*}    
    Next, since $\widehat{\Z}^\nu \subseteq \widehat{\Z}^{\nu_0}$ for all $\nu$, we have 
    \begin{equation*}    
    f_r^\star(\widehat{\Z}^{\nu_0},\cS,\lambda^{\star}) = \max_{\lambda} f_r^\star(\widehat{\Z}^{\nu_0},\cS,\lambda) \geq f_r^\star(\widehat{\Z}^{\nu_0},\cS,\lambda^\nu) \geq f_r^\star(\widehat{\Z}^\nu,\cS,\lambda^\nu),
    \end{equation*} and therefore \begin{equation*}        
    f_r^\star(\widehat{\Z}^{\nu_0},\cS,\lambda^{\star}) \geq \max_{k\le\nu} f_r^\star(\widehat{\Z}^k,\cS,\lambda^k) = LB^\nu. 
    \end{equation*} 
    Thus, we obtain 
    \begin{equation*}
        \lim_{\nu \to \infty} f_r^\star(\widehat{\Z}^{\nu_0},\cS,\lambda^{\star}) - LB^\nu \ge 0.
    \end{equation*}
    On the other hand, $\sum_{s \in \cS} p_s \| \lambda_s^{\nu} - \lambda_s^{\star} \|_2^2 \ge 0$, and hence from \cref{eq:inequality_distance_lambdanu_lambdastar_2} we get 
    \begin{multline}
        \lim_{\nu \to \infty} f_r^\star(\widehat{\Z}^{\nu_0},\cS,\lambda^{\star}) - LB^\nu \hfill\\
        \leq  \lim_{\nu \to \infty} \frac{ \sum_{s \in \cS} p_s \left( \|\lambda_s^{\nu_0} - \lambda_s^{\star} \|_2^2 + \sum_{k=\nu_0}^\nu (\eta^{k})^2 \|  x_s^{k} - \bar{x}^k  \|_2^2 \right)}{2 \sum_{k=\nu_0}^{\nu-1}  \eta^{k}} \\
        \leq  \lim_{\nu \to \infty} \frac{ \sum_{s \in \cS} p_s \left( \|\lambda_s^{\nu_0} - \lambda_s^{\star} \|_2^2 + \text{diam}(\X) \sum_{k=\nu_0}^\nu (\eta^{k})^2  \right)}{2 \sum_{k=\nu_0}^{\nu-1}  \eta^{k}} = 0. \label{eq:limit_UB_gap}
    \end{multline} 
    The last inequality follows from~\cref{assumption:Omega_bounded}. Specifically, since $\X$ is bounded, its diameter, denoted by $\text{diam}(\X)$, is finite and $\text{diam}(\X) = \max_{x_1,x_2 \in \X} \|x_1-x_2\|_2$. Hence, $\| x_s^{k} - \bar{x}^k  \|_2^2 \leq \text{diam}(\X)$. 
    Finally, to derive the limit in \eqref{eq:limit_UB_gap}, note that for all $\nu \ge \nu_0$, we have $\eta^\nu = 1 / (1 + \nu - \nu_0)$. Thus, by construction of the sequence $\eta^\nu$ (see Lines \ref{algstep:reset_eta} and \ref{algstep:update_eta} in \cref{alg:SSCP}), it follows that
    \begin{equation*}        
    \lim_{\nu \to \infty} \sum_{k=\nu_0}^{\nu -1}\eta^k = \lim_{\nu \to \infty} \sum_{k=1}^\nu \frac{1}{k} = \infty
    \end{equation*} 
    and 
    \begin{equation*}    
    \lim_{\nu \to \infty} \frac{\sum_{k=\nu_0}^{\nu -1}(\eta^k)^2}{\sum_{k=\nu_0}^{\nu -1}\eta^k} = \lim_{\nu \to \infty} \frac{\sum_{k=1}^{\nu}(1/k)^2}{\sum_{k=1}^{\nu}1/k} = \frac{\pi^2/6}{\infty}  = 0.
    \end{equation*}
    Consequently, 
    \begin{equation*}   
    \lim_{\nu \to \infty} LB^\nu
    = f_r^\star(\widehat{\Z}^{\nu_0},\cS,\lambda^{\star}).
    \end{equation*} 
\end{proof}


\subsection{Proof of Theorem~\ref{theorem:correctness}}

\begin{restatetheorem}[Restating Theorem~\ref{theorem:correctness}]
\cref{alg:SSCP} terminates after a finite number of iterations and produces a solution $x^\star$ such that \[
\sum_{s \in \mathcal{S}} p_s \, \phi_s^u(x^\star,z_s) \leq f^\star + \epsilon,\] 
where $(x^\star,z)$ satisfies \eqref{eq:upper_level_constraints}-\eqref{eq:bilevel_feasibility}, and $f^\star$ denotes the optimal value of \ref{eq:bilevel_stochastic_program}.
\end{restatetheorem}

\begin{proof}
Suppose $\epsilon=0$. Since $\X$ is finite, \cref{alg:SSCP} can only execute the truth clause in \cref{algstep:CPM_non_recorded_leader_sol} a finite number of times before encountering previously evaluated solutions. Let $(\hat{x}_s^{\nu}, \hat{z}_s^{\nu})$ be the scenario-wise optimal solution to the relaxation \ref{eq:relaxed_bilevel_stochastic_program} computed at iteration $\nu$ (see \cref{alg_step:solve_relaxation}). We establish the following relationships: 
\begin{equation}\label{eq:lower_bound_convergence}
\sum_{s \in \cS} p_s \phi_s^u (\hat{x}_s^{\nu}, \hat{z}_s^{\nu}) \leq \lim_{\nu \to \infty} \sum_{s \in \cS} p_s \phi_s^u (\hat{x}_s^{\nu}, \hat{z}_s^{\nu}) = f_r^\star(\widehat{\Z}^{\nu_0},\cS,\lambda^{\star}) \leq f^\star.
\end{equation}
The first inequality holds because the relaxation \ref{eq:relaxed_bilevel_stochastic_program} becomes progressively more restricted as follower cuts are added over iterations. The subsequent equality is guaranteed by \cref{theorem:convergence}, and the final inequality follows from \cref{proposition:lower_bound}.

By \cref{proposition:non_anticipativity_convergence}, the Lagrangian dual bound $f_r^\star(\widehat{\Z}^{\nu_0},\cS,\lambda^{\star})$ enforces non-anticipativity over the convexified feasible space. However, due to the nonconvex duality gap inherent to mixed-integer programs, the multiplier updates alone may stall at disjoint, fractional scenario solutions ($\bar{x}^\nu \neq x_s^{\nu}$). To overcome this and guarantee primal convergence, \cref{alg:SSCP} actively monitors for a theoretical optimality certificate stall ($0 \in \partial_\lambda f_r^\star$) and invokes a primal space exhaustion mechanism (\cref{algstep:add_no_good_cut}) via no-good cuts on $\X$ whenever the dual bound stalls.

Because $\X$ and $\Z$ are strictly bounded and finite (\cref{assumption:Omega_bounded}), this combined process of dual bounding, follower cut generation, and primal space exhaustion strictly and finitely reduces the search space. Consequently, the algorithm must finitely terminate in one of two ways: either the lower bound surpasses the upper bound (implying the optimal solution is already found and stored in $x^\star$), or the algorithm isolates a primal feasible solution that inherently satisfies discrete non-anticipativity ($\bar{x}^\nu = x_s^{\nu}$ for all $s\in\cS$).
When a consensus solution $x^\nu = \hat{x}_s^{\star,\nu}$ is reached, let $\tilde{z}_s^{\nu}$ be the true optimal follower response obtained by solving \eqref{eq:BSP_follower_problem} (see Lines \ref{algstep:consensus} and \ref{algstep:optimal_follower_response}). Note that $\tilde{z}_s^{\nu} \in \widehat{\Z}^\nu$. Because $(x^\nu, \hat{z}_s^{\nu})$ satisfies the generated optimality cuts \cref{eq:bilevel_feasibility_cuts_EBSP}, and $\bar{x}^\nu = x^\nu$, we have $\phi_s^l (x^\nu,\hat{z}_s^{\nu}) \leq \phi_s^l (x^\nu, \tilde{z}_s^{\nu})$ for all $s \in \cS$. This ensures $\hat{z}_s^{\nu}$ is a true optimal follower response satisfying \eqref{eq:upper_level_constraints}--\eqref{eq:bilevel_feasibility}. Thus, \cref{algstep:set_bilevel_optimal_sol_2} is executed in finite time.

It follows from \eqref{eq:lower_bound_convergence} and the non-anticipativity condition \eqref{eq:nonanticipativity_constraint_EBSP} that:
\begin{equation*}
LB^\nu = \sum_{s \in \cS} p_s \phi_s^u (x_s^{\nu}, \hat{z}_s^{\nu}) \le f^\star \le  \sum_{s \in \cS} p_s \phi_s^u (x^\nu, \hat{z}_s^{\nu}) = \sum_{s \in \cS} p_s \phi_s^u (x_s^{\nu}, \hat{z}_s^{\nu}).
\end{equation*}
The first equality is the definition of the algorithm's lower bound $LB^\nu$ evaluated at the relaxed scenario-wise solution. The first inequality holds because $LB^\nu$ is a valid global lower bound, as established in \eqref{eq:lower_bound_convergence}. The second inequality holds because $x^\nu$, alongside its certified optimal follower response $\hat{z}_s^{\nu}$, constitutes a strictly feasible solution to the original bilevel problem. Therefore, its objective value is a valid upper bound that must be greater than or equal to the true optimum $f^\star=\sum_{s \in \cS} p_s \phi_s^u (x^\star, z_{s})$, where $z$ is a best response to the optimal leader solution $x^\star$. The final equality holds precisely because the algorithm has reached discrete consensus, meaning the scenario-specific decisions $x_s^{\nu}$ are mathematically identical to the consensus decision $x^\nu$ across all scenarios.

Hence, when $\epsilon=0$, \cref{alg:SSCP} returns an optimal solution $x^\star$ to \ref{eq:bilevel_stochastic_program} with an objective value equal to $f^\star$. 

Finally, when $\epsilon>0$, the result naturally follows from the fact that $LB$ and $UB$ are valid bounds, and the algorithm terminates in finite time once their difference is strictly less than~$\epsilon$. 
\end{proof}

\section{
Uncertain Bilevel Knapsack Problem}\label{sec:ubkp}

\looseness-1The Uncertain Bilevel Knapsack Problem (UBKP) extends the classic bilevel knapsack model by introducing uncertain parameters that affect the profit and capacity of both the leader and the follower \citep{Xue2017Uncertain}. In this setting, the leader selects a subset of items to maximize their expected profit, anticipating the follower's response. The follower then selects from a disjoint set of items to maximize their own profit, subject to a knapsack capacity that depends on the leader's selection.

\subsection{\ref{eq:bilevel_stochastic_program} Formulation}

Let $\cS$ be the set of discrete scenarios, with each scenario $s \in \cS$ occurring with probability $p_s$. Let $x \in \{0,1\}^{n_1}$ denote the leader's binary decisions, and $z_s \in \{0,1\}^{n_2}$ denote the follower's binary decisions under scenario $s$. The \ref{eq:bilevel_stochastic_program} formulation of UBKP is defined as follows:
\begin{subequations} \label{eq:ubkp_bsp}
\begin{align}
    \minimize_{x,z} \quad & - \sum_{s \in \cS} p_s \left((d_1^s)^\top x + (d_2^s)^\top z_s\right) \label{eq:ubkp_obj} \\
    \text{subject to} \quad & b_{1}^\top x \leq c_1 \label{eq:ubkp_leader_cap} \\
    & z_s \in \argmin_{z} \left\{ - (d_{3}^s)^\top z \colon (b_1^s)^\top x + (b_2^s)^\top z \leq c_{2}^s \right\}, \quad \forall s \in \cS \label{eq:ubkp_follower} \\
    & x \in \{0,1\}^{n_1}, \quad z \in \{0,1\}^{n_2}
\end{align}
\end{subequations}
where $d_1^s, d_2^s$ are the leader's profit coefficients under scenario $s$, and $d_3^s$ are the follower's profit coefficients. The leader faces a deterministic capacity constraint \eqref{eq:ubkp_leader_cap} with item weights $B_1$ and capacity $c_1$. The follower faces a single lower-level knapsack constraint with stochastic item weights $b_1^s, b_2^s$ and capacity $c_2^s$. Note that we represent the maximization objectives as minimization of the negative profits to align with our framework.

\subsection{\ref{eq:exponential_reformulation_BSP} Reformulation and Big-$M$ Parameters}

To solve Problem~\eqref{eq:ubkp_bsp} with Algorithm~\ref{alg:SSCP}, it needs to be reformulated as \ref{eq:exponential_reformulation_BSP} using scenario-wise value-function cuts. For a given scenario $s$ and a set of historical follower responses $\widehat{\Z}_s$, we enforce optimality by ensuring that the current follower solution $z$ achieves an objective value at least as good as any previously discovered feasible response $\hat{z} \in \widehat{\Z}_s$, provided that $\hat{z}$ is not blocked by the leader's decision $x$. 

The follower's capacity constraints are negated to obtain the $\ge$ form: $-\sum_i b_{1ki}^s x_i - \sum_j b_{2kj}^s z_j \ge -c_{2k}^s$. For a historical response $\hat{z}$, define $\gamma_k(\hat{z}) = \lceil -c_{2k}^s + \sum_j b_{2kj}^s \hat{z}_j \rceil - 1$. The formulation requires two Big-$M$ parameters:
\begin{enumerate}
    \item \textbf{Blocking Parameter ($M$):} The term isolating the leader's variables is $h^{x,l}(x) = -\sum_i b_{1ki}^s x_i$. Since $b_{1ki}^s \ge 0$ and $x_i \ge 0$, the maximum value of this term is $0$. Therefore, we can set a perfectly tight bound of $M = 0$.
    \item \textbf{Dominance Parameter ($\tilde{M}_s$):} If a historical solution $\hat{z}$ is feasible (i.e., not blocked), the current response $z$ must satisfy $-\sum_j d_{3j}^s z_j \le -\sum_j d_{3j}^s \hat{z}_j + \tilde{M}_s \sum_k w_k$, where $w_k \in \{0,1\}$ are the blocking indicator variables. The maximum difference in follower objective values occurs when $\hat{z}$ includes all items and $z$ none. Thus, a valid tight upper bound is $\tilde{M}_s = \sum_j d_{3j}^s$.
\end{enumerate}

\subsection{Tavasl\i o\u glu et al. Reformulation}\label{subsec:ubkp_p4}

We benchmark our decomposition framework against the generalized value-function reformulation of \citet{Tavasliouglu2019SolvingStochastic}, which we call P4, as in the source. In P4, each scenario $s \in \cS$ is treated as an independent follower. Because the follower faces a single knapsack constraint, the quantity the leader passes down to scenario $s$ is the scalar \emph{tender} $(b_1^s)^\top x$, and the follower's value function is the parametric problem
\begin{equation*}
f_s(\beta) = \max_{z} \left\{ (d_3^s)^\top z \colon (b_2^s)^\top z \le \beta \right\}, \qquad s \in \cS,
\end{equation*}
so that $-f_s\big(c_2^s - (b_1^s)^\top x\big)$ is the optimal follower objective in \eqref{eq:ubkp_follower}. Since $b_1^s$ and $b_2^s$ are scenario dependent, every scenario induces its own tender set $\Q_s = \{ (b_1^s)^\top x \colon x \in \X \} = \{ v_1^s, \ldots, v_{\vert \Q_s \vert}^s \}$. Introducing for each scenario a one-hot tender selector $u^s \in \{0,1\}^{\vert \Q_s \vert}$, \eqref{eq:ubkp_bsp} can be reformulated as the single-level mixed-integer program
\begin{subequations}\label{subeq:ubkp_p4}
\begin{align}
    \minimize_{x,z,u} \quad & - \sum_{s \in \cS} p_s \left( (d_1^s)^\top x + (d_2^s)^\top z_s \right) \label{eq:ubkp_p4_obj} \\
    \text{subject to} \quad & b_1^\top x \le c_1, \label{eq:ubkp_p4_leader} \\
    & (b_1^s)^\top x + (b_2^s)^\top z_s \le c_2^s, && s \in \cS, \label{eq:ubkp_p4_follower_feas} \\
    & (b_1^s)^\top x = \sum_{\ell = 1}^{\vert \Q_s \vert} v_\ell^s u_\ell^s, && s \in \cS, \label{eq:ubkp_p4_tender_tie} \\
    & \sum_{\ell = 1}^{\vert \Q_s \vert} u_\ell^s = 1, && s \in \cS, \label{eq:ubkp_p4_onehot} \\
    & (d_3^s)^\top z_s = \sum_{\ell = 1}^{\vert \Q_s \vert} f_s\!\left( c_2^s - v_\ell^s \right) u_\ell^s, && s \in \cS, \label{eq:ubkp_p4_optimality} \\
    & x \in \{0,1\}^{n_1}, \quad z_s \in \{0,1\}^{n_2}, \quad u^s \in \{0,1\}^{\vert \Q_s \vert}, && s \in \cS. \label{eq:ubkp_p4_domains}
\end{align}
\end{subequations}

Constraints \eqref{eq:ubkp_p4_tender_tie}-\eqref{eq:ubkp_p4_onehot} ensure that the selected tender matches the tender realized by $x$, and \eqref{eq:ubkp_p4_optimality} enforces follower optimality by equating the follower's objective to the value function evaluated at the selected tender. The cooperative tie-breaking rule is recovered without additional modeling, as the $d_2^s$ term in \eqref{eq:ubkp_p4_obj} selects the leader-preferred element among the follower's optimal responses. In contrast to \ref{eq:exponential_reformulation_BSP}, formulation \eqref{subeq:ubkp_p4} does not require big-$M$ constants: follower optimality is enforced exactly via the value-function equality \eqref{eq:ubkp_p4_optimality} rather than via indicator-scaled bounds. The trade-off is the upfront construction of the tender sets and the associated value functions $f_s(\cdot)$, which must be completed before the single mixed-integer program can be built.

In this case, the tender set can be constructed without enumerating $\X$. Note that $\Q_s$ requires enumerating the $2^{n_1}$ candidate leader decisions, which is prohibitive at the instance sizes considered here and would understate the performance of P4. This enumeration can be avoided. Since $b_1^s$ is non-negative and integral and $x$ is binary, every realized tender is a non-negative integer. Moreover, the follower is feasible only if $(b_1^s)^\top x \le c_2^s$, because $z = 0$ is the component-wise-minimal follower point. Hence, every tender that a feasible solution of \eqref{subeq:ubkp_p4} can select lies in the grid $\bar{\Q}_s = \{ 0, 1, \ldots, c_2^s \} \supseteq \Q_s \cap \{ v \colon v \le c_2^s \}$. Replacing $\Q_s$ by $\bar{\Q}_s$ leaves the optimal value of problem \eqref{subeq:ubkp_p4} unchanged: a tender in $\bar{\Q}_s \setminus \Q_s$ is unachievable, so \eqref{eq:ubkp_p4_tender_tie} cannot hold for any feasible $x$ when the corresponding selector is active, and such a tender is never selected. We therefore build $\bar{\Q}_s$ directly, which requires $\vert \bar{\Q}_s \vert = c_2^s + 1$ value-function evaluations \emph{independently of $n_1$}. The cost of P4 on the UBKP is thus driven by the follower's capacity rather than by the size of the leader's decision space, and the reformulation remains pseudo-polynomial in $\sum_{s \in \cS} c_2^s$ for this problem class even as $n_1$ grows.

\subsection{Calibration and Computational Environment}\label{subsec:environment}

To generate problem instances, the leader’s deterministic parameters are sampled as follows. The weights are drawn independently as $B_{1i} \sim \mathcal{U}\{1, \ldots, 9\}$, and the capacity is set to $c_1 = \left\lfloor 0.5 \sum_i B_{1i} \right\rfloor$.

Uncertainty across scenarios is modeled using multiplicative noise factors $\xi_1 \sim \mathcal{U}(0.9, 1.4)$, $\xi_2 \sim \mathcal{U}(0.6, 1.05)$, and $\xi_3 \sim \mathcal{U}(0.8, 1.1)$. For each scenario $s \in \cS$, the profit coefficients are generated by scaling the base integer values with these factors:
\begin{equation*}
\begin{aligned}
d_{1i}^s &= \left\lfloor \xi_1 U_{1i} \right\rfloor, \quad &&U_{1i} \sim \mathcal{U}\{2, \ldots, 7\}, \ \forall i, \\
d_{2j}^s &= \left\lfloor  \xi_2 U_{2j} \right\rfloor, \quad &&U_{2j} \sim \mathcal{U}\{1, \ldots, 5\}, \ \forall j, \\
d_{3j}^s &= \left\lfloor  \xi_3 U_{3j} \right\rfloor, \quad &&U_{3j} \sim \mathcal{U}\{2, \ldots, 5\}, \ \forall j.
\end{aligned}
\end{equation*}

The follower’s constraint weights are sampled independently as $b_{1ki}^s, b_{2kj}^s \sim \mathcal{U}\{1, \ldots, 4\}$. For each constraint $k$, the corresponding capacity is defined as
\begin{equation*}
c_{2k}^s = \left\lfloor 0.5 \left( \sum_i b_{1ki}^s + \sum_j b_{2kj}^s \right) \right\rfloor.
\end{equation*}
Finally, the probability weight for each scenario is set uniformly, i.e., $p_s = 1/\vert \cS \vert$.

We varied the number of leader items $n_1 \in \{50,100,200,500\}$, the number of follower items $n_2 \in \{50,100,200,500\}$, and the number of scenarios $\vert \cS\vert \in \{100, 500, 1000\}$, generating a test bed of 48 instances, one for each parameter combination.


Two parameters of the benchmark warrant additional comment, since the comparison is meaningful only if P4 is implemented and parameterized favorably. First, we replaced the value-function equality \eqref{eq:ubkp_p4_optimality} with the inequality $(d_3^s)^\top z_s \ge \sum_\ell f_s(c_2^s - v_\ell^s) u_\ell^s - \varepsilon_{\mathrm{vf}}$, which is equivalent because follower feasibility \eqref{eq:ubkp_p4_follower_feas} already provides the reverse inequality. This avoids floating-point equality checks on non-integral objective coefficients. We use $\varepsilon_{\mathrm{vf}} = 10^{-6}$. This tolerance must remain well below the smallest positive difference between distinct follower objective values, which is of order $10^{-1}$ for the coefficient ranges generated here, so it cannot admit a suboptimal follower response. It should be rescaled together with the objective coefficients if instances of a different magnitude are generated. Second, the time and memory budgets bound both phases of P4 jointly: value-function evaluations are checked against the remaining budget as they are generated, and whatever remains is passed to the single-level mixed-integer program. An instance whose budget is exhausted during the value-function phase yields neither an incumbent nor a bound and is reported with an infinite gap, whereas an instance whose budget is exhausted in the final program is reported with its incumbent and the solver's bound. In contrast, \cref{alg:SSCP} is computationally less expensive, since we initialize the algorithm using only the solution $x^0=\mathbf{0}$.

All numerical experiments were conducted on a 2025 Mac Studio with an Apple M4 Max and 128 GB of RAM, running macOS 26.3.1. The implementation was written in Python 3.13 \citep{python}. Every mixed-integer program arising in either method was solved through Gurobi 13 \citep{gurobi}, and the restricted bundle quadratic program of \cref{alg:SSCP} was solved with CVXPY 1.8.2 \citep{diamond2016cvxpy,agrawal2018rewriting}.

To compare the two approaches, each instance was given a wall-clock limit of 3600 seconds and a peak-memory limit of 10 GB. The solver was restricted to a single thread to maximize reproducibility. Peak memory is measured as the maximum resident set size of the solving process and its spawned solver subprocesses, \emph{in excess of} the baseline recorded immediately before the solve, so that the reported figure reflects the memory attributable to the model rather than the interpreter's footprint.

Gaps are reported in relative terms throughout. For \cref{alg:SSCP}, we report $(UB - LB) / \vert UB \vert$. 
For P4, we report the relative optimality gap returned by the solver for problem \eqref{subeq:ubkp_p4}. A method that terminates without a valid bound because a limit was reached before any bound became available is recorded with an infinite gap. Such instances are omitted from the gap profiles, thereby causing the profiles to plateau below $100\%$.

\subsection{Numerical Results}

\Cref{fig:ubkp_results}(a) reports the relative optimality gap profile, showing, for each method, the percentage of the 48 instances whose final gap does not exceed the value on the horizontal axis. The two curves cross, and neither method dominates in every respect. At tight tolerances, P4 is clearly ahead: it certifies $2.1\%$ of instances to within $0.1\%$ and reaches $22.9\%$ within a $1\%$ gap, whereas \cref{alg:SSCP} does not produce gaps below $1\%$ for any instance. Above a gap of roughly $4\%$, the ordering reverses: \cref{alg:SSCP} places $47.9\%$ of instances within $5\%$ and $77.1\%$ within $10\%$, against $35.4\%$ and $37.5\%$ for P4. The clearest difference is at the right end of the range, where the profile of \cref{alg:SSCP} reaches $100\%$ while the P4 profile plateaus at $81.2\%$: for 9 of the 48 instances, P4 exhausted its time or memory budget before any bound became available, whereas the decomposition returns a valid bound for every instance. This asymmetry is structural rather than incidental: the scenario decomposition maintains valid bounds from its first iteration and can be stopped at any time, whereas P4 produces no bound until its value-function phase is complete and the solver finds a feasible solution. The corresponding cost is a heavier tail: the largest relative gap returned by \cref{alg:SSCP} is $11.07$, against $0.74$ across the 39 instances on which P4 terminates with a bound, so on the hardest instances the Lagrangian bound can remain weak even though it is always available.

\begin{figure}[htbp]
\centering
\subfigure[Relative optimality gap profile]{\includegraphics[width=0.48\textwidth]{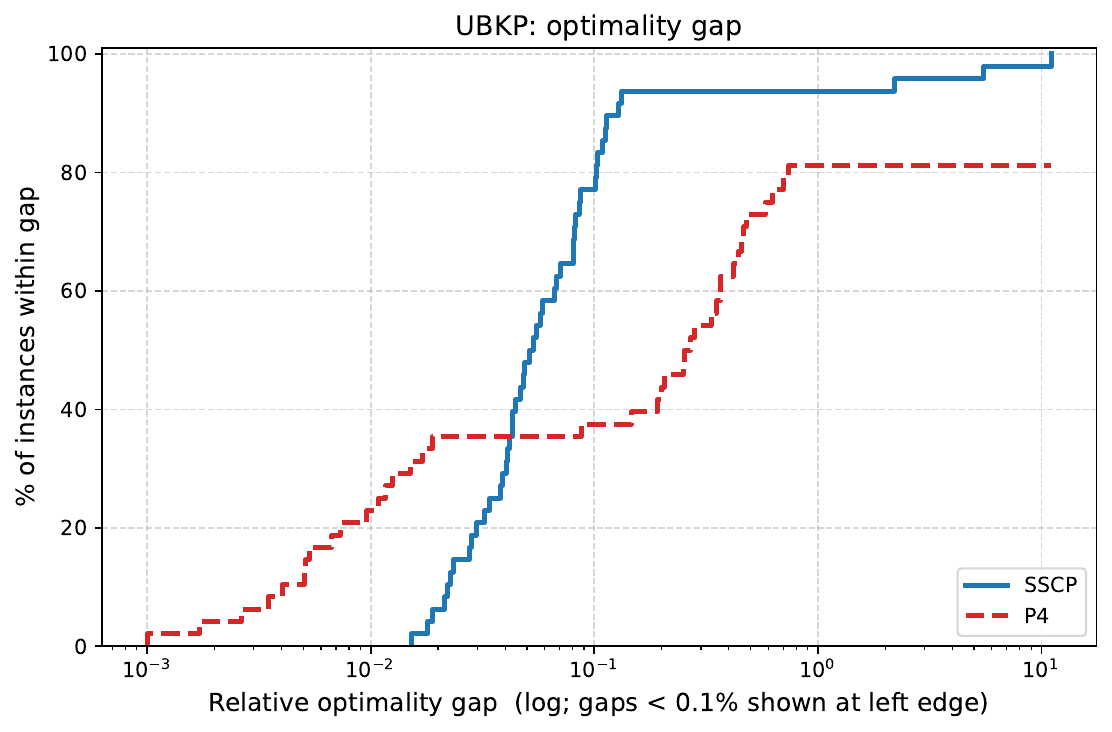}}
\hfill
\subfigure[Gap versus solution time]{\includegraphics[width=0.48\textwidth]{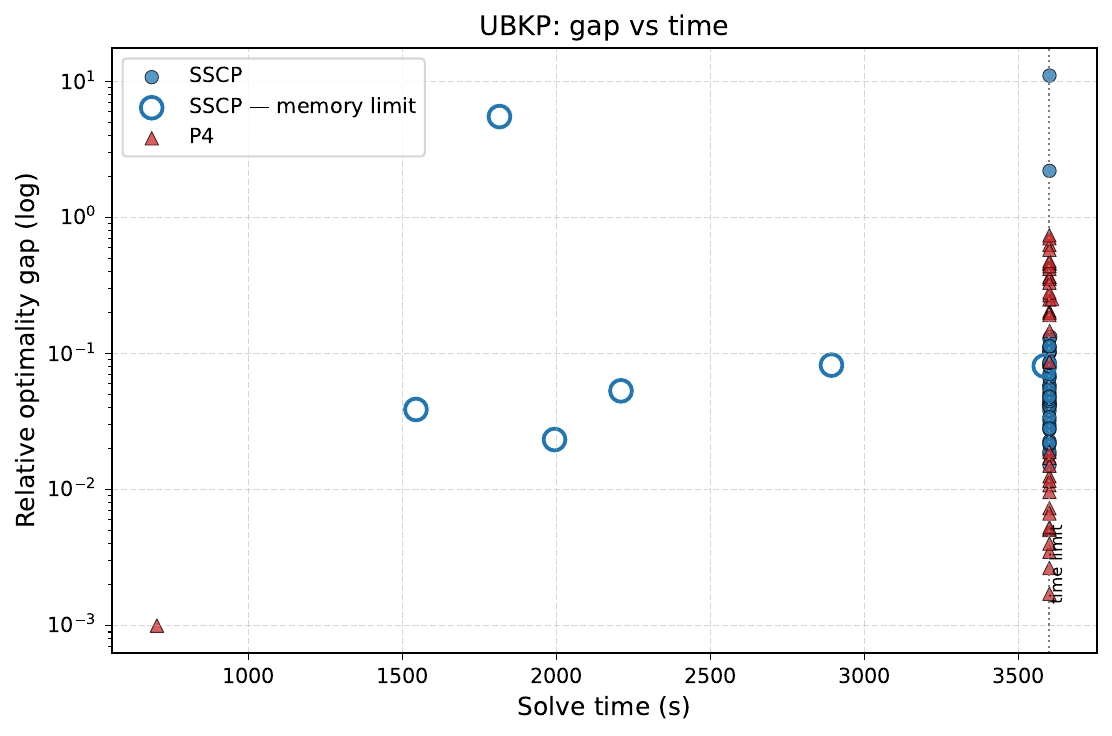}}
\\[0.5em]
\subfigure[Peak memory profile]{\includegraphics[width=0.48\textwidth]{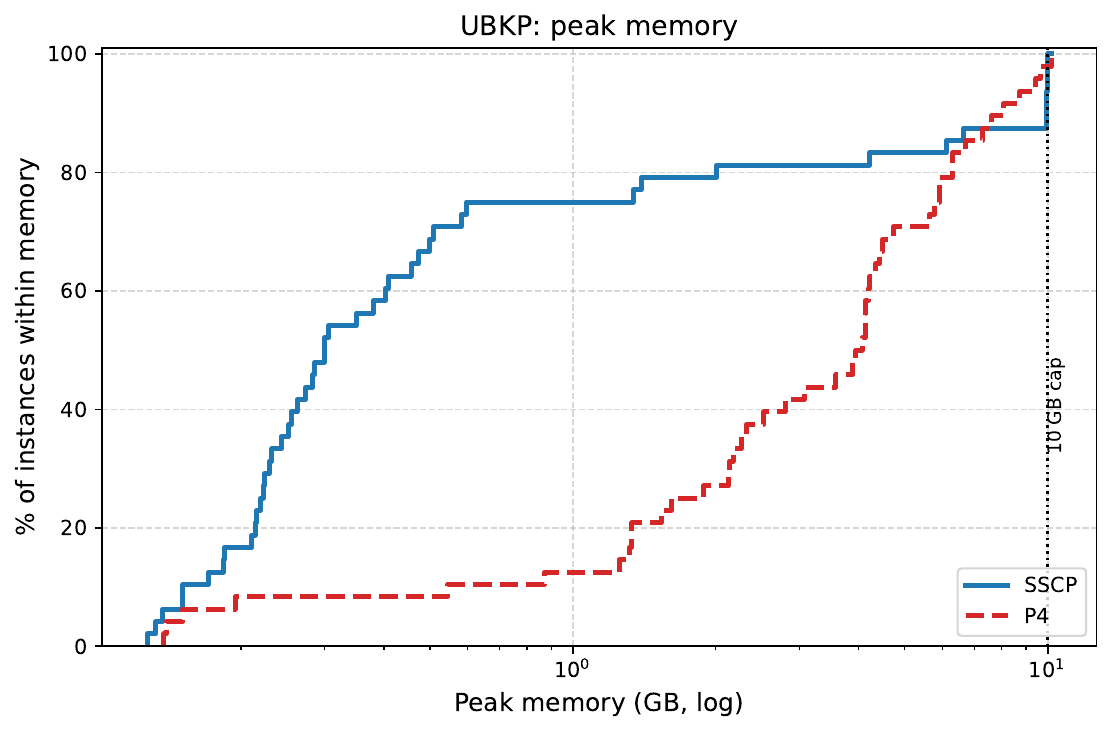}}
\hfill
\subfigure[Work per instance]{\includegraphics[width=0.48\textwidth]{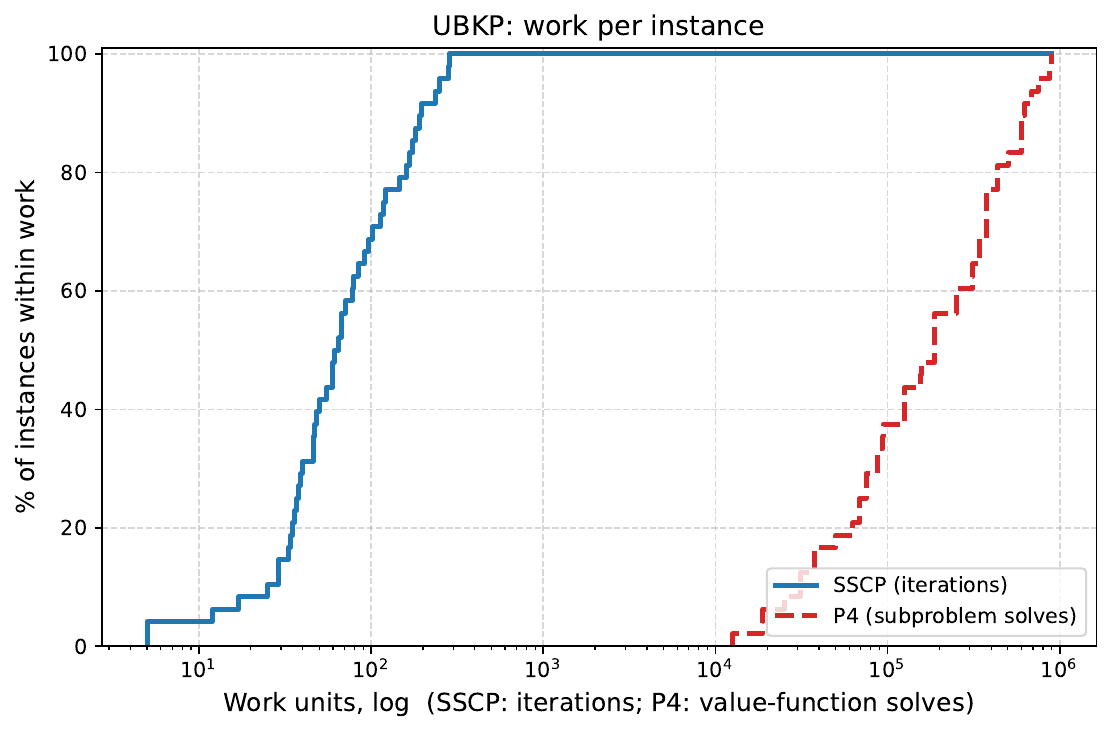}}
\caption{UBKP results across the 48 instances. In panel (a), the curve plateaus below $100\%$ by exactly the share of instances for which that method returned no valid bound. In panel (b), hollow markers denote runs terminated by the memory limit, and instances without a bound are omitted. Panels (c) and (d) include every run, whether or not it produced a bound. In panel (d), the two series represent different work units: an iteration of \cref{alg:SSCP} comprises $O(\vert\cS\vert)$ mixed-integer solves, whereas a P4 unit is a single value-function solve.}
\label{fig:ubkp_results}
\end{figure}

\Cref{fig:ubkp_results}(b) plots the same gaps against solution time and shows how the budget is consumed. At this scale, the comparison is of gap quality at a fixed budget rather than of time to optimality: 42 of the 48 runs of \cref{alg:SSCP} and 44 of the 48 runs of P4 exhaust the one-hour limit. The six remaining runs of \cref{alg:SSCP}, shown as hollow markers, are stopped by the 10 GB memory limit after between $1{,}543$ and $3{,}583$ seconds. Because \cref{alg:SSCP} maintains a valid bound at every iteration, each of these runs still returns a finite gap, with five below $9\%$ and one at $5.53\%$. P4 exhibits the opposite failure mode: one instance is solved to proven optimality in $702$ seconds, but its three memory-limited runs are stopped during the value-function phase and, never having built the mixed-integer program, return no bound at all. The six time-limit runs whose value-function phase was still incomplete after one hour also return no bound.

\begin{table}[htbp]
\centering\footnotesize
\caption{UBKP: outcome distribution and matched-subset comparison over the 48 instances. Gaps are relative; memory is peak usage attributable to the run.}
\begin{threeparttable}
\label{tab:ubkp_results}
\begin{tabular}{lrr}
\toprule
 & \cref{alg:SSCP} & P4 \\
\midrule
\multicolumn{3}{l}{\emph{Panel A: outcomes over all 48 instances}} \\
Instances returning a valid bound & 48 & 39 \\
Gap $\le 0.1\%$ & 0.0\% & 2.1\% \\
Gap $\le 1\%$ & 0.0\% & 22.9\% \\
Gap $\le 5\%$ & 47.9\% & 35.4\% \\
Gap $\le 10\%$ & 77.1\% & 37.5\% \\
Largest finite gap & 11.075 & 0.741 \\
Largest peak memory (GB) & 10.01 & 10.18 \\
Instances at the memory cap & 12.5\% & 2.1\% \\
Optimal / time limit / memory limit & 0 / 42 / 6 & 1 / 44 / 3 \\
\midrule
\multicolumn{3}{l}{\emph{Panel B: mean relative gap, by matched subset}} \\
Both returned a bound ($n=39$) & 0.5330 & 0.2144 \\
Only \cref{alg:SSCP} returned a bound ($n=9$) & 0.0751 & --- \\
Only P4 returned a bound ($n=0$) & --- & --- \\
Each over its own bounded set & 0.4471\tnote{a} & 0.2144\tnote{b} \\
\midrule
\multicolumn{3}{l}{\emph{Panel C: mean peak memory (GB), by matched subset}} \\
Both hit the time limit ($n=39$) & 0.539 & 3.555 \\
Only \cref{alg:SSCP} hit the time limit ($n=3$) & 3.710 & 5.050 \\
Only P4 hit the time limit ($n=5$) & 9.966 & 5.223 \\
\midrule
\multicolumn{3}{l}{\emph{Panel D: instances proved optimal}} \\
Both proved optimality & \multicolumn{2}{c}{none} \\
Only P4 proved optimality ($n=1$): memory (GB) & 0.263 & 0.544 \\
\hspace{1em}time (s) & 3600.2 & 702.1 \\
\bottomrule
\end{tabular}
\begin{tablenotes}\footnotesize
\item[a] Mean over all 48 instances. \item[b] Mean over the 39 instances on which P4 returned a bound.
\end{tablenotes}
\end{threeparttable}
\end{table}

\Cref{fig:ubkp_results}(c) reports the peak-memory profile. The two curves have markedly different shapes. The profile of \cref{alg:SSCP} rises steeply, with a median of $0.30$ GB and $75\%$ of instances below $1$ GB, but develops a heavy tail on the hardest instances as the accumulated follower-response sets $\widehat{\Z}_s$ grow, and $12.5\%$ of runs terminate at the cap. The P4 profile is shifted right, with a median of $4$ GB and $87.5\%$ of instances above $1$ GB, since the tender grids and the value functions over them must be held in full before the single-level mixed-integer program is assembled; only one P4 run reaches the cap, so on this problem class memory is a binding resource for the decomposition on a few extreme instances and a uniformly elevated but rarely binding cost for P4. \Cref{fig:ubkp_results}(d) reports work per instance in each method's native unit, and the two series are therefore not directly comparable: \cref{alg:SSCP} performs between 5 and 284 iterations (median 63), each comprising $O(\vert\cS\vert)$ mixed-integer solves, whereas P4 performs between $1.3 \times 10^4$ and $8.9 \times 10^5$ value-function solves per instance, in line with the pseudo- polynomial size $\vert \bar{\Q}_s \vert = c_2^s + 1$ of the tender grids.

\Cref{tab:ubkp_results} decomposes these aggregates and guards against a survivorship confound in the profile comparison. Restricted to the 39 instances on which both methods return a bound (Panel B), the ordering of mean gaps reverses: $0.533$ for \cref{alg:SSCP} against $0.214$ for P4. The reversal is explained by the nine instances that only \cref{alg:SSCP} bounds, which are comparatively easy for it (mean gap $0.075$): removing them leaves the decomposition holding the hardest instances alone, so coverage and average quality must be read together rather than separately. Panels C and D refine the resource picture. On the 39 instances where both methods run out the clock, \cref{alg:SSCP} uses on average an order of magnitude less memory ($0.54$ against $3.56$ GB), and the five instances on which only P4 hits the time limit are precisely the instances on which \cref{alg:SSCP} was stopped by the memory cap (mean peak $9.97$ GB). The single instance solved to proven optimality within the budget is solved by P4.

\section{
Bilevel Fuel Treatment Planning (BiFTP)}
\label{sec:biftp}

The Bilevel Fuel Treatment Planning (BiFTP) model addresses the optimal spatial allocation of fuel treatments to mitigate wildfire risk \citep{LAGOS2025bilevel} and is formulated as a Stackelberg game. The fuel treatment planner (the leader) determines treatment locations and types before the fire season. Nature (the follower) then designs a worst-case wildfire attack over multiple periods to maximize expected damage across a finite set of weather scenarios.

\subsection{\ref{eq:bilevel_stochastic_program} Formulation}

Let $V$ be the set of geographical areas, $\mathcal{R}_i$ the set of available treatments in area $i$ (including the \emph{no treatment} option), and $\mathcal{T}$ the set of time periods. The leader's decision variable $x_{ir} \in \{0,1\}$ indicates whether treatment $r$ is applied in area $i$. Nature's decision variable $z_{ir}^{ts} \in \{0,1\}$ indicates whether a fire occurs in area $i$ under treatment $r$ in period $t$ under scenario $s$. The \ref{eq:bilevel_stochastic_program} formulation is as follows:

\begin{subequations} \label{eq:biftp_bsp}
\begin{align}
    \min_{x} \max_{z} \quad & \sum_{s \in \cS} p_s \sum_{i \in V,t \in \mathcal{T},r \in \mathcal{R}_i} a_{ir}^{ts} z_{ir}^{ts} \label{eq:biftp_obj} \\
    \text{subject to} \quad & \sum_{r \in \mathcal{R}_i} x_{ir} = 1, && i \in V \\
    & \sum_{i \in V,r \in \mathcal{R}_i} c_{ir} x_{ir} \le E \\
    & \sum_{i \in V_k,r \in \mathcal{R}_i} z_{ir}^{ts} \le b_k^{ts}, && k \in \mathcal{K}, t \in \mathcal{T}, s \in \cS \\
    & \sum_{t \in \mathcal{T}} z_{ir}^{ts} \le x_{ir}, && i \in V, r \in \mathcal{R}_i, s \in \cS \label{eq:biFTP_protection} \\
    & x_{ir} \in \{0,1\}, \quad z_{ir}^{ts} \in \{0,1\}
\end{align}
\end{subequations}
where $a_{ir}^{ts}$ is the damage incurred, $c_{ir}$ is the treatment cost, and $E$ is the total treatment budget. Nature's attacks are constrained geographically by overlapping clusters $V_k \in \mathcal{K}$, and $b_k^{ts}$ limits the number of simultaneous fires in cluster $k$.

\subsection{\ref{eq:exponential_reformulation_BSP} Reformulation and Big-$M$ Parameters}

Due to the tight coupling between the leader's and follower's decisions ($z_{ir}^{ts} \le x_{ir}$), standard Big-$M$ reformulations yield weak relaxations. When the leader alters a treatment $x_{ir}$, nature's previous attack variable $z_{ir}^{ts}$ is trivially invalidated, thereby removing the associated cut.

To formulate a robust \ref{eq:exponential_reformulation_BSP}, we exploit the structure of nature's constraints. Nature's budget $b_k^{ts}$ limits the \emph{areas} that can be attacked, irrespective of the \emph{treatment type} the leader implements. 
We can aggregate any follower decision $\hat{z}$ across treatment types to extract the spatio-temporal attack pattern $\hat{v}_{it}(\hat{z}) = \sum_{r \in \mathcal{R}_i} \hat{z}_{ir}^{ts}$, which is a binary indicator equal to one if and only if $\hat{z}$ attacks area $i$ in period $t$.
Nature can deploy the identical attack pattern $\hat{v}_{it}$ against any new leader decision $x$ (up to a projection of this attack onto the feasible space, see \eqref{eq:biFTP_protection}) and still satisfy every cluster budget $b_k^{ts}$, because the cluster sums depend only on the areas attacked, not on the treatment types. Therefore, instead of using Big-$M$ variables to block historical responses, we generate global Benders-like structural cuts. The leader's damage under scenario $s$ is bounded below by evaluating the historical attack pattern against the current treatment selection:
\begin{equation}
    \text{Damage}^s \ge \sum_{i \in V,r \in \mathcal{R}_i,t \in \mathcal{T}} a_{ir}^{ts} x_{ir} \hat{v}_{it}(\hat{z}), \quad \forall \hat{z} \in \hat{\Z}_s.
\end{equation}
Since $\sum_{r \in \mathcal{R}_i} x_{ir} = 1$, the summation over $r$ acts as a selector: the right side sums the damage to each attacked area under the \emph{current} plan $x$. The cut is valid because any permuted attack pattern is a feasible, though not necessarily optimal, response to the leader's decision $x$, so the worst-case damage $\text{Damage}^s$ can only be larger. This formulation eliminates the need for large Big-$M$ parameters, mitigating numerical instability and forcing the leader to immediately account for geographic vulnerabilities, regardless of the chosen treatment.

\subsection{Tavasl\i o\u glu et al. Reformulation}\label{subsec:biftp_p4}

We also cast the BiFTP model in the P4 form of \citet{Tavasliouglu2019SolvingStochastic}. Recall that nature acts as the follower, and the planner's decisions enter nature's problem only through the coupling constraint \eqref{eq:biFTP_protection}. The tender is therefore the treatment plan itself. Because all scenarios share the cluster incidence structure and the coupling matrix, a single family of value functions suffices, and the tender set is common to all scenarios, $\Q = \X = \{ x^{(1)}, \ldots, x^{(\vert \Q \vert)} \}$, where $\X$ collects the plans satisfying $\sum_{r \in \mathcal{R}_i} x_{ir} = 1$ for all $i \in V$ and $\sum_{i \in V, r \in \mathcal{R}_i} c_{ir} x_{ir} \le E$. Nature's value function on a plan is
\begin{equation*}
g_s(x) = \max_{z} \left\{ \sum_{i \in V, r \in \mathcal{R}_i, t \in \mathcal{T}} a_{ir}^{ts} z_{ir}^{ts} \colon \quad\begin{aligned}
\sum_{i \in V_k, r \in \mathcal{R}_i} z_{ir}^{ts} \le b_k^{ts} \ \forall k \in \mathcal{K}, t \in \mathcal{T}; \\ 
\quad \sum_{t \in \mathcal{T}} z_{ir}^{ts} \le x_{ir} \ \forall i \in V, r \in \mathcal{R}_i
\end{aligned} 
\right\}.
\end{equation*}
With a one-hot selector $u \in \{0,1\}^{\vert \Q \vert}$ over plans, \eqref{eq:biftp_bsp} becomes

\begin{subequations}\label{subeq:biftp_p4}
\begin{align}
    \minimize_{x,z,u} \quad & \sum_{s \in \cS} p_s \sum_{i \in V, r \in \mathcal{R}_i, t \in \mathcal{T}} a_{ir}^{ts} z_{ir}^{ts} \\
    \text{subject to} \quad & \sum_{r \in \mathcal{R}_i} x_{ir} = 1, && i \in V, \\
    & \sum_{i \in V, r \in \mathcal{R}_i} c_{ir} x_{ir} \le E, \\
    & \sum_{i \in V_k, r \in \mathcal{R}_i} z_{ir}^{ts} \le b_k^{ts}, && k \in \mathcal{K}, t \in \mathcal{T}, s \in \cS, \\
    & \sum_{t \in \mathcal{T}} z_{ir}^{ts} \le x_{ir}, && i \in V, r \in \mathcal{R}_i, s \in \cS, \\
    & x_{ir} = \sum_{\ell = 1}^{\vert \Q \vert} x_{ir}^{(\ell)} u_\ell, && i \in V, r \in \mathcal{R}_i, \\
    & \sum_{\ell = 1}^{\vert \Q \vert} u_\ell = 1, \\
    & \sum_{i \in V, r \in \mathcal{R}_i, t \in \mathcal{T}} a_{ir}^{ts} z_{ir}^{ts} = \sum_{\ell = 1}^{\vert \Q \vert} g_s\!\left( x^{(\ell)} \right) u_\ell, && s \in \cS, \label{eq:biftp_p4_value_function} \\
    & x_{ir} \in \{0,1\}, \quad z_{ir}^{ts} \in \{0,1\}, \quad u \in \{0,1\}^{\vert \Q \vert}.
\end{align}
\end{subequations}

The value-function equality \eqref{eq:biftp_p4_value_function} is load-bearing here in a way it is not for the UBKP: because the planner has no direct objective term, this equality is the sole link between the chosen plan and the realized adversarial damage, and dropping it would allow the planner to report zero damage. As with the UBKP, this reformulation is free of big-$M$ parameters, and the value-function equality \eqref{eq:biftp_p4_value_function} is instrumental in avoiding them.
The trade-off, however, is far more severe than in \cref{subsec:ubkp_p4}. Because the tender coincides with the plan, no analog of the grid construction is available: the tender set is the plan set itself, $ \Q  = \X $. Hence, P4 must enumerate all budget-feasible plans and evaluate $g_s$ on each. The number of plans grows exponentially in $\vert V \vert$, so this enumeration, rather than the final mixed-integer program, becomes the binding constraint on tractability. Our scenario decomposition never enumerates $\X$, and its per-iteration work is independent of $\vert \X \vert$.

\subsection{Parameter Generation}

To generate the parameters, we consider three treatment types: No-Treatment (NT) with zero cost, Prescribed Burning (PB) with cost $c_{ir} = 1$, and Thinning-from-Below (TFB) with cost $c_{ir} = 5$.

The damage parameters $a_{ir}^{ts}$ are intended to reflect treatment effectiveness. For each area $i$, period $t$, and scenario $s$, a base damage value is sampled as $D_i^{ts} \sim \mathcal{U}(10, 100)$. Treatment-specific damages are then defined as follows:
\begin{itemize}    
\item NT: $a_{i,\text{NT}}^{ts} = D_i^{ts}$,
\item PB: $a_{i,\text{PB}}^{ts} =  \alpha_i^{ts} D_i^{ts} $, where $\alpha_i^{ts} \sim \mathcal{U}(0.5, 0.8)$,
\item TFB: $a_{i,\text{TFB}}^{ts} =  \beta_i^{ts} D_i^{ts} $, where $\beta_i^{ts} \sim \mathcal{U}(0.1, 0.4)$.
\end{itemize}

The overlapping clusters $\mathcal{K}$ are constructed by sampling, for each cluster $k \in \mathcal{K}$, a subset $V_k \subseteq V$ whose cardinality is drawn from $\mathcal{U}\{\lceil 0.3 \vert V \vert \rceil, \ldots, \lfloor 0.7 \vert V \vert \rfloor\}$. The attack budget bounds $b_k^{ts}$ for nature are sampled from $\mathcal{U}\{1,\dots,\lfloor 0.5 \vert V_k\vert \rfloor\}$ for each cluster $k$, period $t$, and scenario $s$. Finally, scenario probabilities are set to be equal, i.e., $p_s = 1 / \vert \cS \vert$.

We varied the number of geographical areas $\vert V \vert \in \{10, 20, 50, 70, 100\}$, setting the treatment budget $E = \vert V \vert$, the number of overlapping clusters $\vert \mathcal{K} \vert \in \{1, 20, 50\}$, and the number of scenarios $\vert \cS \vert \in \{100, 500, 1000\}$, over two time periods, generating a test bed of 45 instances with increasing spatial complexity, one per parameter combination. The computational environment, resource limits, and benchmark calibration are identical to those described in \cref{subsec:environment}.

\subsection{Numerical Results}

\Cref{fig:biftp_results}(a) reports the relative optimality gap profile, and the contrast with the UBKP is stark, reflecting the structural difference identified in \cref{subsec:biftp_p4}. The profile of \cref{alg:SSCP} rises steadily and reaches $100\%$: every instance is returned with a valid bound, $13.3\%$ are solved to within $0.1\%$, $37.8\%$ attain a gap below $1\%$, $64.4\%$ fall within $10\%$, and the largest gap over the whole test bed is $0.475$. In contrast, the P4 profile is flat at $6.7\%$ across the entire range because P4 produces a bound in only 3 of the 45 instances. On those three, the reformulation is exact and certifies optimality, but on the remaining 42, either the plan enumeration or the value-function phase exhausts the budget before the mixed-integer program can be built, so neither an incumbent nor a bound is available. At these spatial scales, the comparison is therefore not one of gap quality but of usability: the benchmark either solves an instance outright or returns nothing, whereas the scenario decomposition returns a usable bound on all 45.

\begin{figure}[htbp]
\centering
\subfigure[Relative optimality gap profile]{\includegraphics[width=0.48\textwidth]{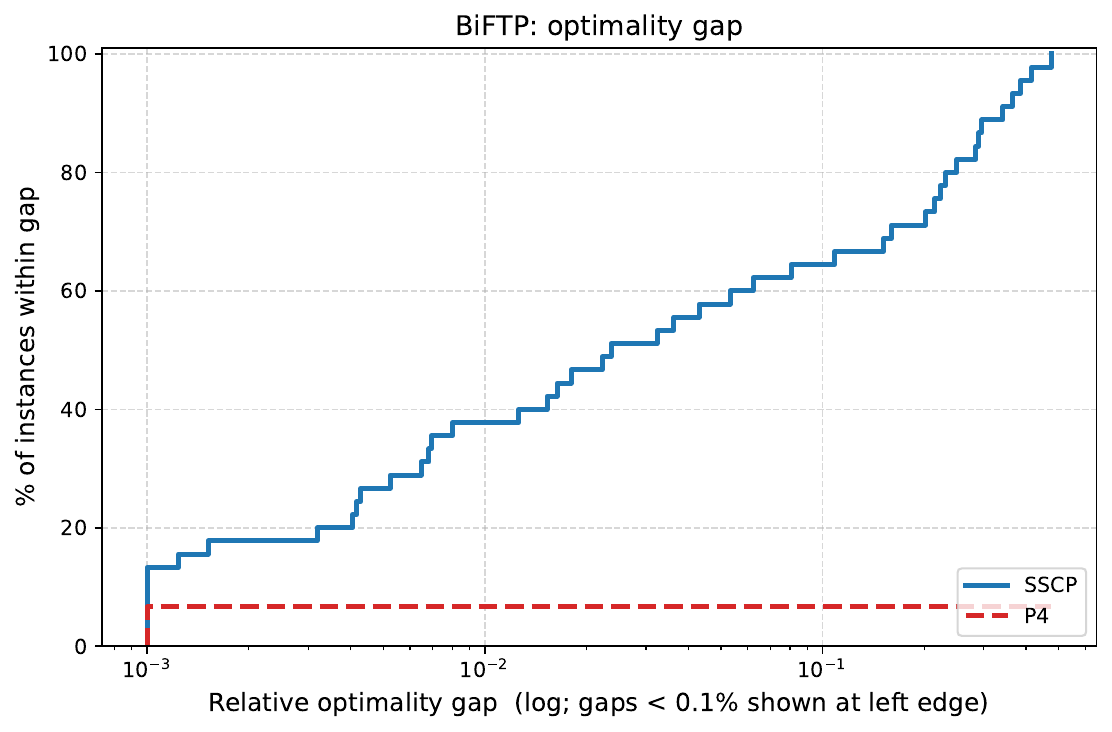}}
\hfill
\subfigure[Gap versus solution time]{\includegraphics[width=0.48\textwidth]{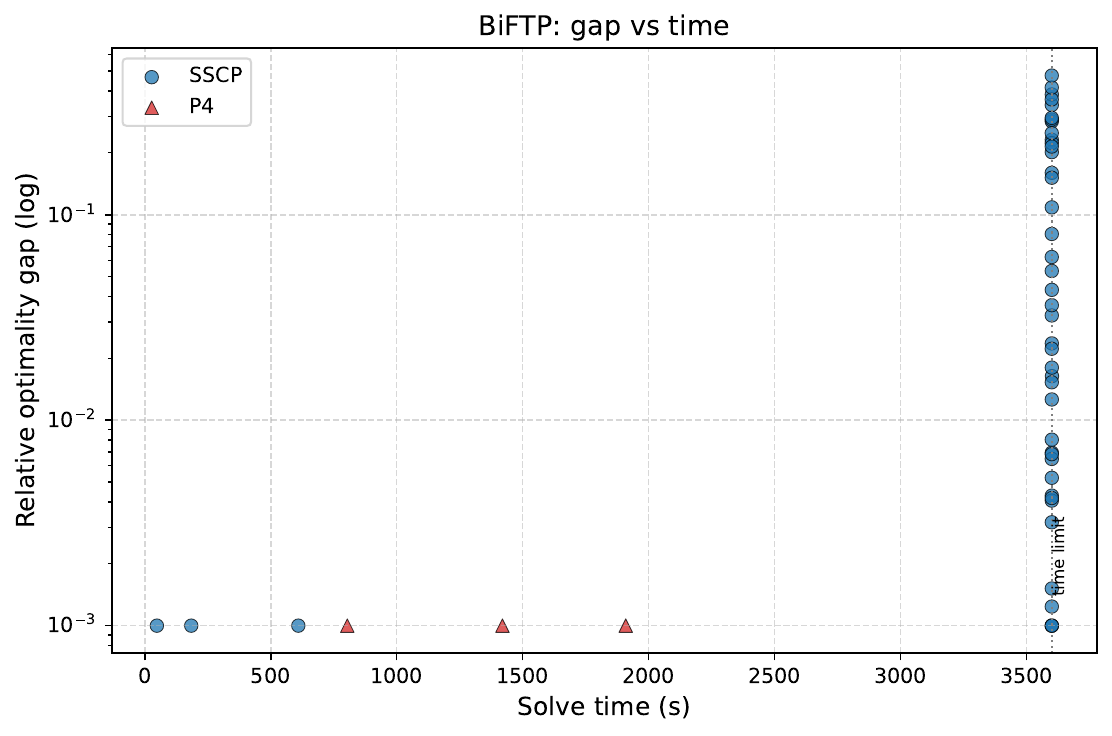}}
\\[0.5em]
\subfigure[Peak memory profile]{\includegraphics[width=0.48\textwidth]{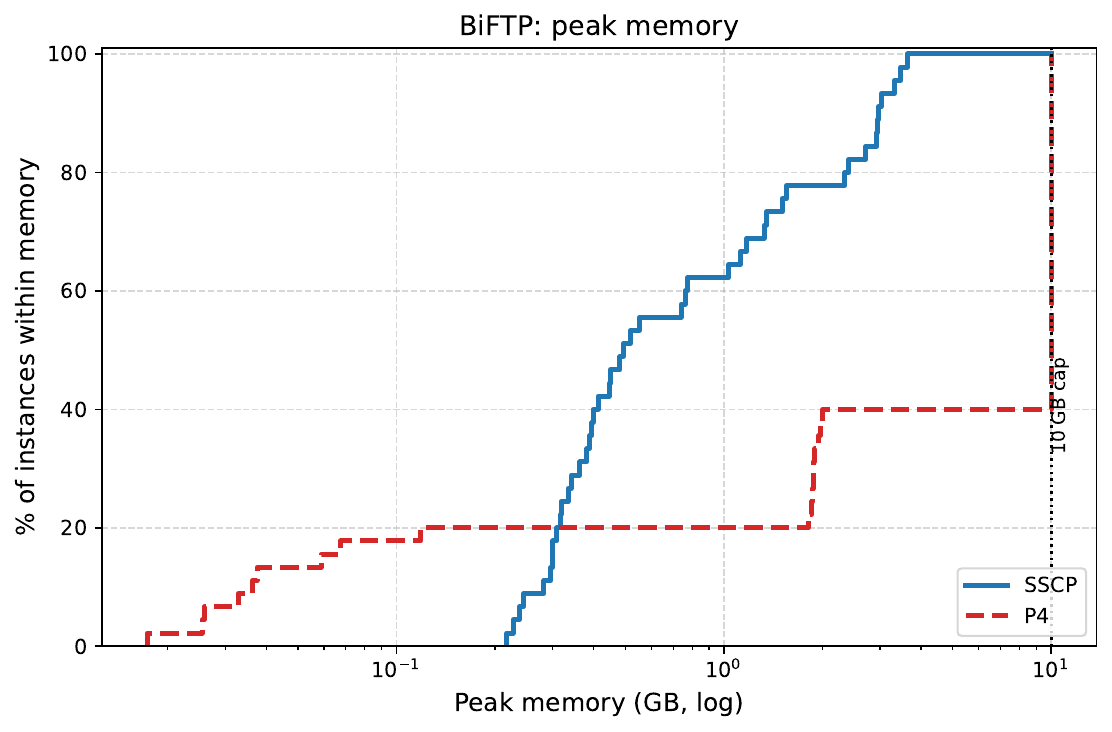}}
\hfill
\subfigure[Work per instance]{\includegraphics[width=0.48\textwidth]{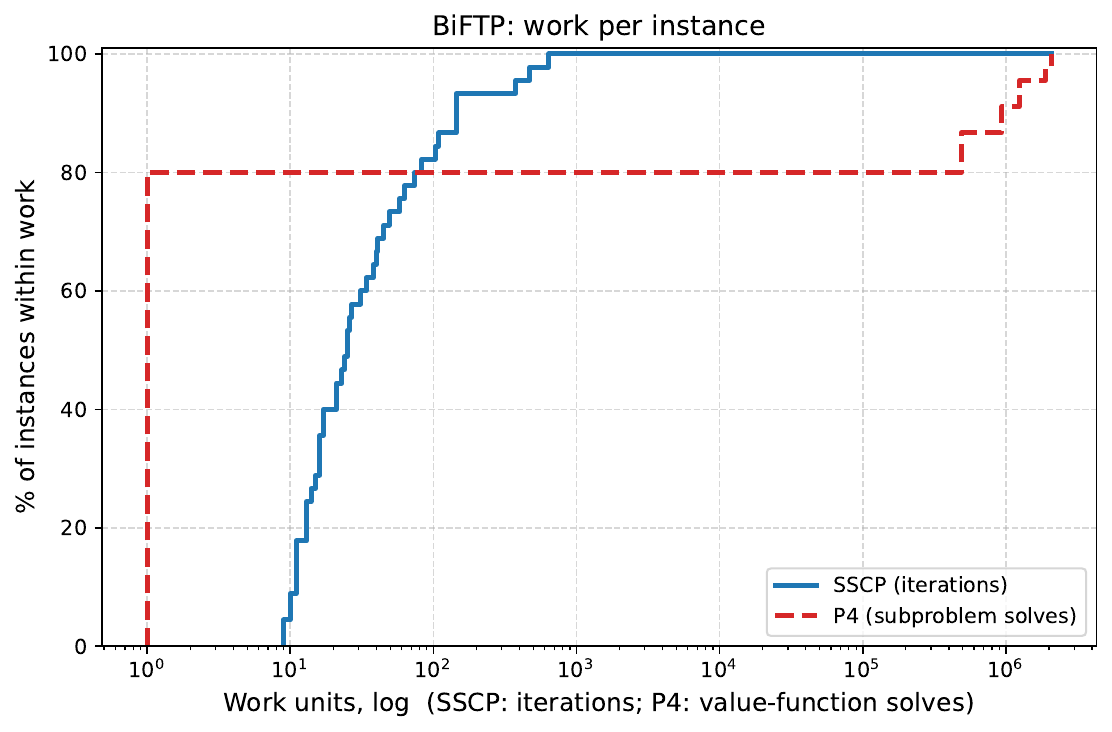}}
\caption{BiFTP results across the 45 instances. Conventions follow \cref{fig:ubkp_results}. The dotted vertical line in panel (c) marks the $10$ GB budget imposed on both methods. In panel (d), the sharp spike on the left in P4 represents the 36 instances terminated during plan enumeration before completing a single value-function solve.}
\label{fig:biftp_results}
\end{figure}

\looseness-1\Cref{fig:biftp_results}(c) isolates the principal mechanism behind P4's failure by reporting peak memory. The memory consumed by \cref{alg:SSCP} remains modest across the entire test bed, never exceeding $3.64$ GB, because the scenario subproblems are solved one at a time and only the accumulated follower responses $\widehat{\Z}_s$ are retained between iterations. The P4 curve behaves quite differently: it piles up against the 10 GB limit, at which $60\%$ of the instances terminate. The tender set of \cref{subsec:biftp_p4} must be constructed in full, together with one value function per plan and scenario, so memory, rather than time, is the binding resource for the majority of these instances. An additional 15 instances reach the time limit during tender set construction or the value-function phase. This is precisely the regime in which \citet{Tavasliouglu2019SolvingStochastic} anticipate that their reformulation will struggle, and it is the regime for which our decomposition is designed.

\begin{table}[htbp]
\centering\footnotesize
\caption{BiFTP: outcome distribution and matched-subset comparison over the 45 instances. Conventions as in \cref{tab:ubkp_results}.}
\begin{threeparttable}
\label{tab:biftp_results}
\begin{tabular}{lrr}
\toprule
 & \cref{alg:SSCP} & P4 \\
\midrule
\multicolumn{3}{l}{\emph{Panel A: outcomes over all 45 instances}} \\
Instances returning a valid bound & 45 & 3 \\
Gap $\le 0.1\%$ & 13.3\% & 6.7\% \\
Gap $\le 1\%$ & 37.8\% & 6.7\% \\
Gap $\le 5\%$ & 57.8\% & 6.7\% \\
Gap $\le 10\%$ & 64.4\% & 6.7\% \\
Largest finite gap & 0.475 & 0.000 \\
Largest peak memory (GB) & 3.64 & 10.00 \\
Instances at the memory cap & 0.0\% & 60.0\% \\
Optimal / time limit / memory limit & 3 / 42 / 0 & 3 / 15 / 27 \\
\midrule
\multicolumn{3}{l}{\emph{Panel B: mean relative gap, by matched subset}} \\
Both returned a bound ($n=3$) & 0.0007 & 0.0000 \\
Only \cref{alg:SSCP} returned a bound ($n=42$) & 0.1157 & --- \\
Only P4 returned a bound ($n=0$) & --- & --- \\
Each over its own bounded set & 0.1080\tnote{a} & 0.0000\tnote{b} \\
\midrule
\multicolumn{3}{l}{\emph{Panel C: mean peak memory (GB), by matched subset}} \\
Both hit the time limit ($n=13$) & 1.980 & 1.326 \\
Only \cref{alg:SSCP} hit the time limit ($n=29$) & 0.760 & 9.313 \\
Only P4 hit the time limit ($n=2$) & 1.053 & 0.063 \\
\midrule
\multicolumn{3}{l}{\emph{Panel D: instances proved optimal}} \\
Both proved optimality ($n=1$): memory (GB) & 0.227 & 0.119 \\
\hspace{1em}time (s) & 48.3 & 804.1 \\
Only \cref{alg:SSCP} ($n=2$): memory (GB) & 1.053 & 0.063 \\
\hspace{1em}time (s) & 397.2 & 3602.0 \\
Only P4 ($n=2$): memory (GB) & 3.212 & 0.027 \\
\hspace{1em}time (s) & 3600.2 & 1664.5 \\
\bottomrule
\end{tabular}
\begin{tablenotes}\footnotesize
\item[a] Mean over all 45 instances. \item[b] Mean over the 3 instances on which P4 returned a bound.
\end{tablenotes}
\end{threeparttable}
\end{table}

\Cref{fig:biftp_results}(b) makes the all-or-nothing behavior visible on a per-instance basis. Because runs without a bound are omitted, only three P4 points appear, at $804$, $1420$, and $1909$ seconds. These are the three $\vert V \vert = 10$ instances with $\vert \cS \vert = 100$, on which the plan set is sufficiently small ($\vert \Q \vert = 4889$) for the enumeration, the value-function phase, and the final program to complete within the budget. \Cref{alg:SSCP}, in turn, proves optimality on the three $\vert V \vert = 10$, $\vert \mathcal{K} \vert = 1$ instances after $48$, $185$, and $610$ seconds, and every one of its remaining runs sits at the one-hour line with a finite gap. The P4 failures are, moreover, decided early: the median memory-limited run is stopped after roughly 12 minutes, so in most instances the benchmark's outcome is determined long before the hour elapses.

\Cref{fig:biftp_results}(d) reports work per instance and explains why successes are so scarce. \Cref{alg:SSCP} performs between 9 and 642 iterations (median 25), a count driven by the difficulty of the Lagrangian dual rather than by $\vert \X \vert$, since its per-iteration work never enumerates the plan set. Each P4 success, by contrast, requires roughly $4.9 \times 10^5$ value-function solves at $\vert V \vert = 10$. Runs that fail at the time limit accumulate up to $2.1 \times 10^6$ solves without completing the phase. For $\vert V \vert \ge 20$, enumeration itself exhausts the memory budget before a single value function is computed, as shown by the sharp spike on the left in the figure.

\Cref{tab:biftp_results} completes the picture. The matched subsets of Panel B are nearly degenerate, as only the three easiest instances are bounded by both methods, and on them both are essentially exact; the informative row is the 42 instances bounded only by \cref{alg:SSCP}, on which its mean gap is $0.116$. Panel C carries the memory story in matched form: across the 29 instances where only \cref{alg:SSCP} hits the time limit, it averages $0.76$ GB, compared to $9.31$ GB for P4, which, on those instances, is stopped by the memory cap instead. Panel D records that each method proves optimality on exactly three instances, but the two sets share only one instance, on which \cref{alg:SSCP} is roughly $17$ times faster ($48$ against $804$ seconds). For the two instances where only \cref{alg:SSCP} certifies optimality, P4 stands at the time limit without a bound. Conversely, for the two instances where only P4 certifies optimality, \cref{alg:SSCP} terminates at the time limit with an open gap.

\section{Concluding remarks}\label{sec:concluding_remarks}

This paper presents a novel exact decomposition framework for solving mixed-integer bilevel stochastic programs, a class of problems known to be $\Sigma_2^p$-hard. By integrating value-function reformulations with a scenario-wise Lagrangian decomposition, we establish a mathematically rigorous stochastic subgradient cutting-plane algorithm. We prove that this approach converges to the global optimum in a finite number of iterations.

Benchmarking against the generalized value-function reformulation of \citet{Tavasliouglu2019SolvingStochastic} under a common time and memory budget shows that the two approaches are complementary on the UBKP, where tender sets admit a compact grid construction. It also shows that our decomposition is the only method that returns usable bounds on the BiFTP, where the reformulation must enumerate the leader's plan set and is stopped by the memory budget on most instances.

Several promising research directions emerge from our work. For example, extending this framework to the pessimistic formulation of the bilevel stochastic problem \ref{eq:bilevel_stochastic_program}, in which the follower selects the optimal response that inflicts maximum damage on the leader, would require a novel approach to representing the lower-level optimality conditions.

\bibliographystyle{apalike} 
\bibliography{bib} 

\end{document}